\documentclass{article}
\usepackage{graphicx,color}
\usepackage{stmaryrd,amsmath,amsfonts,amssymb,amsthm,mathrsfs,amsopn}
\usepackage{latexsym}
\usepackage[numbers]{natbib}
\usepackage{hyperref}
\usepackage[usenames,dvipsnames,x11names,table]{xcolor}

\numberwithin{equation}{section}
\usepackage[utf8]{inputenc}
\usepackage{tikz}
\usepackage{esint}

\def\H{\mathcal H}

\def\R{\mathbb R}
\def\N{\mathbb N}
\def\Z{\mathbb Z}
\newcommand{\dist}{\mathop{\mathrm{dist}}}

\def\00{{\bf 0}}

\def\Lip{{\rm Lip}\,}

\newcommand{\hel} {
\hskip2.5pt{\vrule height7pt width.5pt depth0pt}
\hskip-.2pt\vbox{\hrule height.5pt width7pt depth0pt}
\, }
\newcommand{\restr}{\hel}
\newcommand{\supp}{\mathrm{spt}}
\theoremstyle{plain}
\newtheorem{theorem}{Theorem} [section]

\newtheorem{lemma}[theorem]{Lemma}

\theoremstyle{definition}

\theoremstyle{plain}
\newtheorem{prop}[theorem]{Proposition}
\theoremstyle{definition}
\newtheorem{defin}[theorem]{Definition}
\theoremstyle{plain}
\newtheorem{cor}[theorem]{Corollary}
\renewcommand{\epsilon}{\varepsilon}
\renewcommand{\phi}{\varphi}
\renewcommand{\bar}{\overline}
\title{Lectures on Existence and Regularity of Anisotropic Minimal Surfaces}
\author{Antonio De Rosa}
\date{}
\begin{document}
\maketitle
\begin{center}
Lecture notes for the mini-course given by the
author at the Winter School\\
in Geometric Measure Theory at Westlake University in January 2025.
\end{center}
\begin{abstract}
For several natural phenomena, the use of the surface area functional is a first approximation. In order to capture microstructures, numerous models in applied sciences employ directionally dependent functionals, known as anisotropic energies. Since anisotropic energies are not invariant under rigid motions, their critical points do not enjoy the same conservation laws as isotropic minimal surfaces. For instance, the monotonicity formula for the density ratio is not known to hold for minimizers of general anisotropic energies. Consequently, the study of anisotropic minimal surfaces is more challenging than the study of their isotropic counterparts. In this mini-course, we give an overview of the state of the art in the existence and regularity theory of anisotropic minimal surfaces. In particular, we first focus on solutions of the anisotropic Plateau problem and subsequently move to the investigation of the rectifiability and regularity theory for critical points of anisotropic energies. To conclude, we provide applications to the min-max theory for the construction of closed optimally regular anisotropic minimal hypersurfaces in closed Riemannian manifolds.

\end{abstract}
\section{Introduction}
Recall that if $\Sigma \subset \mathbb{R}^n$ is a $d$-dimensional (isotropic) minimal surface, then the quantity
\[
\frac{\mathcal{H}^d(\Sigma\cap B(x,r))}{r^d}
\]
is non-decreasing as a function of $r>0$, see for instance \cite{SimonBook}. This fact is known as the \emph{monotonicity formula for area}. The monotonicity formula is a fundamental tool in the local analysis of minimal surfaces, implying for example that every blow-up of a minimal surface is a tangent cone and that the density
$$\Theta^d(x,\Sigma) = \lim_{r\to 0} \frac{\mathcal{H}^d(\Sigma\cap B(x,r))}{\omega_dr^{d}}$$
is a well-defined and upper semicontinuous function in $x$, where $\omega_d=\mathcal{H}^d(B^d(0,1))$ is the volume of the unit $d$-ball.
Many key regularity results for minimal surfaces, like Allard regularity theory \cite{All3}, rely on such monotonicity.
As we will shortly see, this fundamental tool fails for anisotropic minimal surfaces \cite{All1}. 

Let us first introduce anisotropic minimal surfaces. Let $G(n,d)$ denote the Grassmannian of $d$-dimensional linear subspaces of $\mathbb{R}^n$. Fix a function
\[ F: \mathbb{R}^n \times G(n,d) \to (0,\infty) \]
of class $C^1$. We define the anisotropic energy of a $d$-dimensional surface $\Sigma \subset \mathbb{R}^n$ by
\[ \mathbf{F}(\Sigma) := \int_{\Sigma} F(x, T_x \Sigma)\, d\mathcal{H}^d(x), \]
where $T_x \Sigma$ is the approximate tangent $d$-plane at $x$. In what follows, we will simplify notation by assuming $F$ does not depend on $x$, which can be done without loss of generality for many local considerations. In particular, by the compactness of the Grassmannian and the continuity and positivity of $F$, there exist $\lambda, \Lambda>0$ such that
\begin{equation}\label{ell}
0< \lambda \le F \le \Lambda.
\end{equation}
In the codimension-one case $d=n-1$, it is convenient to identify $F$ with an even function defined on the unit sphere $\mathbb S^{n-1}$.
A surface $\Sigma$ is called an \emph{anisotropic minimal surface} with respect to $F$ if for every one-parameter family of diffeomorphisms $\{\varphi_t\}_{t\in(-\epsilon,\epsilon)}$ of the ambient space with $\varphi_0 = \mathrm{Id}$ and $(\varphi_t-\mathrm{Id})$ compactly supported, one has
\[ \left.\frac{d}{dt}\right|_{t=0} \mathbf{F}(\varphi_t(\Sigma)) = 0. \]
In other words, $\Sigma$ is a critical point of the functional $\mathbf{F}$ under compactly supported variations.

The motivation to consider anisotropic surface energies comes from materials science, where surface tension often depends on the orientation of the interface between different materials. Such energies cannot be modeled by the standard isotropic area functional.
A basic difficulty is that for a general $F$ there is no known monotonicity formula for density ratios of anisotropic minimal surfaces. In fact, Allard proved that in general one can expect a monotonicity formula to hold only for anisotropic integrands that are linear transformations of the area integrand  \cite{All1}. We illustrate this below, by a simple calculation for energy minimizers.
If one replaces a portion of a surface by a cone competitor in a ball $B(x,r)$ as in the picture below, one can compare the surface energy of the original surface to that of the cone.

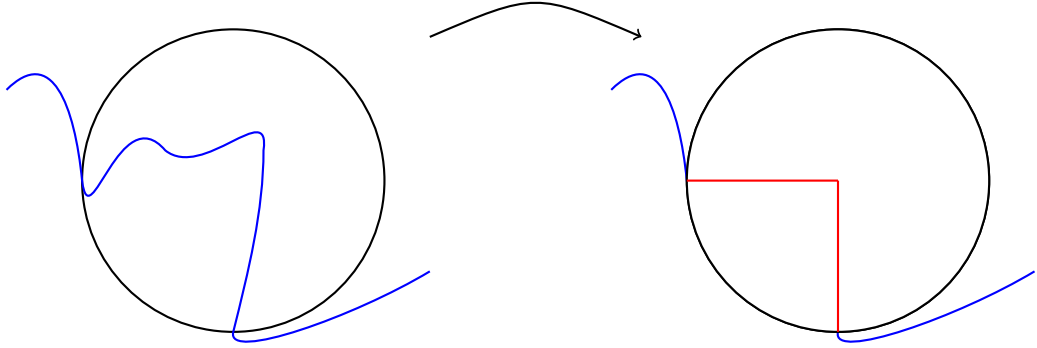
\begin{center}
\begin{tikzpicture}[scale=2]
\def\Rr{1}
\coordinate (C1) at (0,0);
\coordinate (C2) at (4,0);
\draw[thick] (C1) circle (\Rr);
\draw[blue,thick]
  (-1.50,0.60)
    .. controls (-1.20,0.90) and (-1.05,0.50) .. (-1,0)     
    .. controls (-0.95,-0.40) and (-0.75,0.55) .. (-0.45,0.20)
    .. controls (-0.20,0.00) and ( 0.25,0.55) .. ( 0.20,0.20)
    .. controls ( 0.20,-0.20) and ( 0.10,-0.60) .. ( 0,-1)  
    .. controls (-0.05,-1.20) and ( 0.80,-0.90) .. ( 1.30,-0.60);
\draw[thick,->] (1.30,0.95) .. controls (2.00,1.25) .. (2.70,0.95);
\draw[thick] (C2) circle (\Rr);
\begin{scope}[shift={(C2)}]
  \draw[blue,thick]
    (-1.50,0.60)
      .. controls (-1.20,0.90) and (-1.05,0.50) .. (-1,0)
      .. controls (-0.95,-0.40) and (-0.75,0.55) .. (-0.45,0.20)
      .. controls (-0.20,0.00) and ( 0.25,0.55) .. ( 0.20,0.20)
      .. controls ( 0.20,-0.20) and ( 0.10,-0.60) .. ( 0,-1)
      .. controls (-0.05,-1.20) and ( 0.80,-0.90) .. ( 1.30,-0.60);
\end{scope}
\fill[white] (C2) circle (\Rr);
\draw[thick]  (C2) circle (\Rr);
\coordinate (P) at (3,0);      
\coordinate (Q) at (4,-1);     
\draw[red,thick] (C2) -- (P);
\draw[red,thick] (C2) -- (Q);
\end{tikzpicture}
\end{center}

For the isotropic integrand $F\equiv 1$, this argument yields for a minimizer $K$, and at a.e. $r>0$
\[
f(r):=\mathcal{H}^{d}(K\cap B(x,r)) \leq \mathcal{H}^{d}(\text{Cone} \cap B(x,r)) = \frac r{d} \mathcal{H}^{d-1}( K \cap \partial B(x,r))\leq \frac r{d} f'(r), \]
from which one deduces
\[ \frac{d}{dr}\Big[ \log f(r) - d \log r \Big] \ge 0, \]
which in turn implies, thanks to the monotonicity of the exponential, that
$$\frac{d}{dr} \left (\frac{f(r)}{r^d} \right)  \ge 0.$$
 This is the monotonicity formula for area. We refer the reader to \cite{DLGM,DPDRG1} for more details.

 However, for a general anisotropic integrand $F$, the cone competitor argument produces extra factors. Indeed, by \eqref{ell}, one finds
\begin{multline*}
 f(r) = \mathcal{H}^d(K \cap B(x,r)) \le \frac{1}{\lambda} \mathbf{F}(K \cap B(x,r))
\le \frac{1}{\lambda} \mathbf{F}(\text{Cone} \cap B(x,r)) \le \frac{\Lambda}{\lambda} \mathcal{H}^d(\text{Cone}\cap B(x,r)) = \frac{\Lambda}{\lambda} \frac r{d} f'(r).
\end{multline*}
This leads to a weaker inequality
\[ \frac{d}{dr} \left ( \frac{f(r)}{r^{d\,\lambda/\Lambda}} \right) \ge 0, \]
which is not as useful, since the denominator does not have the correct dimensional homogeneity \cite{DLDRG,DPDRG3}. 

In general, no simple monotonicity formula is known for anisotropic minimizers~\cite{All1}. Hence, one needs to develop new tools that can be used in place of monotonicity. We refer to \cite{DR1} for a survey on anisotropic minimal surfaces.

For instance, let us consider the \emph{anisotropic Plateau problem} in codimension one, in the class of finite perimeter sets (or rectifiable currents). More precisely, suppose $d=n-1$ and identify $F$ with a \emph{uniformly convex norm} on $\mathbb{R}^n$, that is $F$ depends on the unit normal of the hypersurfaces. Fix an open set $\Omega \subset \mathbb{R}^n$ and a set of finite perimeter $G$. Among all finite-perimeter sets $E$ satisfying the ``boundary condition'' $E \setminus \Omega = G \setminus \Omega$, one can attempt to minimize the anisotropic surface energy $\mathbf{F}(\partial^* E)$ (here $\partial^* E$ denotes the reduced boundary of $E$). By the direct method in the calculus of variations, it turns out that there exists a minimizer $E$ in this class, and its boundary $\Sigma = \partial^* E$ is an $F$-minimal surface solving the anisotropic Plateau problem. Moreover, a regularity theorem due to Almgren, Schoen, and Simon \cite{SSA} proves the \emph{almost optimal regularity} $\mathcal{H}^{n-3}(\mathrm{Sing}(\Sigma))=0$ for $\Sigma$. The lack of a monotonicity formula is bypassed by means of competition arguments and by relying on the stability inequality.

 The condition $\mathcal H^{n-3}(\mathrm{Sing}(\Sigma))=0$ is almost sharp: Morgan constructed a counterexample, that is, an anisotropic energy-minimizing hypersurface in $\mathbb{R}^4$ with a singular point~\cite{Morgan1}. Morgan's example is the cone over the Clifford torus $\mathbb{S}^1\times \mathbb{S}^1 \subset \mathbb{R}^4$, which has one singular point at the vertex of the cone, and which he proves to be a calibrated anisotropic minimal surface with respect to anisotropic integrands that are smooth and uniformly elliptic perturbations of the $L^\infty$-norm. The Work-Raccoon Theorem of White \cite[Section 5]{W2} guarantees the existence of $\varepsilon(F)>0$ such that $\mathcal H^{n-3-\varepsilon}(\mathrm{Sing}(\Sigma))=0$. However it is still an open question whether there are anisotropic energy minimizers with a singular set of fractal dimension in the interval $(n-4,n-3-\varepsilon)$, as the lack of a monotonicity formula does not allow one to stratify the singular set.
 
The regularity theory for the anisotropic Plateau problem in general codimension is much less understood, and the best regularity result available provides zero $\mathcal{H}^d$-measure of the singular set of $d$-dimensional minimizers in $\mathbb R^n$ \cite{Alm3}.
In summary, due to the absence of a monotonicity formula, many standard techniques for minimal surfaces do not carry over to anisotropic minimal surfaces, hence the need to develop new tools to build a robust anisotropic minimal surface theory.
\vspace{0.3cm}

{\bf Acknowledgment.} The author Antonio De Rosa was funded by the European Union: the European Research Council (ERC), through StG ``ANGEVA'', project number: 101076411. Views and opinions expressed are however those of the author only and do not necessarily reflect those of the European Union or the European Research Council. Neither the European Union nor the granting authority can be held responsible for them.
These lecture notes are based on the mini-course the author gave at the Winter School in Geometric Measure Theory at Westlake University in January 2025. The author thanks Ioann Vasilyev for his suggestions. The author also thanks his postdocs Daniele De Gennaro, Aria Halavati and Ling Wang for proofreading the final version of these lecture notes.
\section{Notation and preliminaries}
\subsection{Varifolds}
We recall that a set $K\subset \mathbb{R}^n$ is said to be countably $\mathcal{H}^d$-rectifiable if there exists a countable collection $\{f_i\}_{i\in \mathbb N}$ of continuously differentiable maps
$$f_i:\mathbb{R}^d \to \mathbb{R}^n$$
such that
$$\mathcal{H}^d\left (K\setminus \bigcup_{i\in \mathbb N} f_i\left(\mathbb{R}^d\right)\right )=0.$$
An important property of rectifiable sets is that there is a notion of approximate tangent space $T_x K$, which is defined for $\mathcal{H}^d$-almost every $x\in K$, see \cite{SimonBook}. This allows one to extend to $K$ the definition of $\mathbf{F}(K)$.
We now recall some definitions and fundamental results from the theory of varifolds (generalized surfaces), specialized to the anisotropic setting. A \emph{$d$-varifold} in $\mathbb{R}^n$ is a positive Radon measure $V$ on $\mathbb{R}^n \times G(n,d)$, intuitively representing a ``surface'' together with distributional information about its tangent planes. An important class is that of rectifiable varifolds.
A $d$-varifold $V$ is rectifiable if
\[ dV(x,S) = \theta(x)\, d(\mathcal{H}^d \restr K)(x)  d\delta_{T_x K}(S), \]
which we will often express in the more compact form
\[ V = \theta\, \mathcal{H}^d \restr K \;\otimes\; \delta_{T_x K}, \]
where $K$ is a countably $\mathcal{H}^d$-rectifiable set in $\mathbb{R}^n$ and the multiplicity $\theta(x)$ is a locally $\mathcal{H}^d$-integrable nonnegative function on $K$. That is, $V$ assigns to $\mathcal{H}^d$-almost every $x\in K$ the weight $\theta(x)$ and the approximate tangent plane $T_x K$.  If $\theta\equiv 1$, we will often write, by abuse of notation, $V=K$.
We denote by $\|V\|=\pi_\# V$ the weight measure of $V$ in $\mathbb{R}^n$, where \(\pi: (x,T)\in \mathbb{R}^n\times G(n,d)\to x\in \mathbb{R}^n\) and define the lower and upper density respectively as
$$
\Theta_*^d(x,V):= \liminf_{r\to 0^+} \frac{\|V\|(B(x,r))}{\omega_d r^d} , \qquad
\Theta^{d,*}(x,V):= \limsup_{r\to 0^+} \frac{\|V\|(B(x,r))}{\omega_d r^d} .
$$
If the two limits coincide, we call their common value the density and denote it by $\Theta^d(x,V)$.
Given a $C^1$ diffeomorphism $\psi:\R^n\to\R^n$, we define the varifold $\psi_\#V$ such that for every $\Phi\in C^0_c(\mathbb{R}^n\times G(n,d))$
$$\int_{\mathbb{R}^n\times G(n,d)}\Phi(x,T)d(\psi_\#V)(x,T)=\int_{\mathbb{R}^n\times G(n,d)}\Phi(\psi(x),d_x\psi(T))J\psi(x,T) dV(x,T).$$
Here $d_x\psi(T)$ is the image of $T$ under the map $d_x\psi$ and
\[
J\psi(x,T):=\sqrt{\det\Big(\big(d_x\psi\big|_T\big)^*\circ d_x\psi\big|_T\Big)}
 \]
 denotes the $d$-Jacobian determinant of the differential $d_x\psi$ restricted to the $d$-plane $T$. If $V= \mathcal{H}^d \restr \Sigma \otimes \delta_{T_x \Sigma}$  is the varifold canonically associated to a smooth $d$-dimensional submanifold $\Sigma$, the pushforward $\psi_\#V$ is simply the image $\psi(\Sigma)$ of $\Sigma$ under the map $\psi$.

Let  \(\eta^{x,r}(y)=(y-x)/r\), then we call tangent measures of a measure $\mu$ at $x$ all limits of the following form:
$$
\mu_{x,r_j}:=\frac{1}{\mu(B(x,r_j))} (\eta^{x,r_j}_\#\mu)\restr {B(0,1)} \rightharpoonup \sigma \in {\rm Tan}(x, \mu).
$$
Analogously, we call tangent varifolds of a varifold $V$ at $x$ all the following limits:
\begin{equation*}
V_{x,r_j}:=\frac{r_j^d}{\|V\|(B(x,r_j))}\big((\eta^{x,r_j})_\#V  \big)\restr {B(0,1)\times G(n,d)}\rightharpoonup W \in {\rm Tan}(x, V).
\end{equation*}
Comparing the definition of tangent varifolds with the one of tangent measures, the additional factor $r_j^d$ compensates for the presence of the Jacobian determinant in the definition of $(\eta^{x,r_j})_\#V$.
Moreover, we note that in defining tangent measures and tangent varifolds we rescale by the mass. This will be useful to guarantee the existence of such limits and that they are non-trivial, as the lack of a monotonicity formula prevents us from obtaining density bounds.
\subsection{Anisotropic integrands}
Given an \emph{anisotropic integrand} $F \in C^1(G(n,d))$, we can extend the anisotropic energy to any varifold $V$ as follows
$$
\mathbf{F}(V):= \int_{\R^n\times G(n,d)} F(T) d V(x,T).
$$
We can then define the anisotropic first variation of a varifold $V$ as
\[
\delta_FV(g) := \left.\frac{d}{d\varepsilon}\right|_{\varepsilon = 0}\mathbf{F}((\mathrm{Id}+\varepsilon g)_\#(V)), \quad \forall g \in C^1_c(\mathbb{R}^n,\mathbb{R}^n).
\]
One can compute \cite[Appendix A]{DPDRG2} that the anisotropic first variation of $V$ with respect to $F$ is the order-one distribution whose
action on $g \in C_c^1(\mathbb{R}^n, \mathbb{R}^n)$ is given by
\[ \delta_F V(g) = \int_{\mathbb{R}^n \times G(n,d)} B_F(T): Dg(x) \, dV(x,T). \]
Here $B_F(T) \in \mathbb{R}^{n\times n}$ is the explicit matrix defined as
$$
B_F(T) = F(T)T + T^\perp dF(T)T, \qquad \mbox{where }\quad dF(T) := DF(T) +DF(T)^t,
$$
$DF(T)$ denotes the differential of $F$ once extended to a $C^1$ function defined in a small neighborhood of $G(n,d)$ in $\mathbb{R}^{n\times n}$, we identify the $d$-plane $T$ with the orthogonal projection of $\mathbb{R}^n$ onto $T$, and we denote by $A : B:=\operatorname{tr}(A^TB)$ the Frobenius scalar product between matrices.

We say that $\delta_F V$ is a Radon measure if this first variation is represented by integration against some vector-valued Radon measure.
We say that $V$ has $F$-mean curvature in $L^p$ if there exists a map $H\in L^p(\|V\|;\R^n)$ such that
\begin{equation}\label{boundedmean}
\delta_F V(g) = -\int_{\mathbb R^n}\langle H(x),g(x) \rangle d\|V\|(x), \quad \forall g \in C_c^1(\mathbb{R}^n, \mathbb{R}^n).
\end{equation}
If $H$ can be chosen to be $0$, that is $\delta_F V = 0$, we say that $V$ is $F$-\emph{stationary}. For instance, if $\Sigma$ is a smooth $F$-minimal surface, then $V = \mathcal{H}^d\restr \Sigma \;\otimes\; \delta_{T_x \Sigma}$ is $F$-stationary.

Ultimately, in order to prove a rectifiability theorem for $F$-stationary varifolds, we will need to show that $V$ has atomic tangent plane distributions, in the sense that at $\|V\|$-almost every point $x$ there is a unique tangent plane, which in turn identifies an invariant space for the blow-up varifolds at $x$. This leads to the key notion of the \emph{atomic condition} for $F$. To define this ellipticity condition, one can show that, for $\|V\|$-almost every $x$, any tangent varifold $W$ of $V$ at $x$ has the form
\[ dW(y,T) = d\sigma(y) \otimes d\mu_x(T) \]
for some measure $\sigma$ on $\mathbb{R}^n$ and some probability measure $\mu_x$ on $G(n,d)$. Intuitively, $\mu_x$ describes the distribution of tangent planes of $V$ at $x$. The $F$-stationarity of $W$ implies
\[ 0 = \delta_F W(g) = \int  B_F(T)\!:\!Dg(y) \, dW(y,T) = \int_{\mathbb{R}^n} A_F( \mu_x ): Dg(y)\, d\sigma(y) \]
for all $g\in C^1_c(\mathbb{R}^n,\mathbb{R}^n)$, where
\[ A_F(\mu_x) := \int_{G(n,d)} B_F(T)\, d\mu_x(T) \in \mathbb{R}^{n\times n}. \]
 By integration by parts, the above condition yields $A_F(\mu_x) D\sigma = 0$ in the sense of distributions, which can be interpreted as saying $\sigma$ is invariant along all directions spanning $\mathrm{Im}(A_F(\mu_x)^T)$. Denote $W_x := \mathrm{Im}(A_F(\mu_x)^T)$, which is a linear subspace of dimension $\dim W_x = n-\dim(\mbox{Ker} A_F(\mu_x))$. $F$-stationarity thus forces $\sigma$ to be invariant under translations along $W_x$. If $\dim W_x<d$, then $W$ may concentrate on sets of dimension strictly smaller than $d$, which is not possible for rectifiable varifolds. This motivates the following definition we introduced in collaboration with De Philippis and Ghiraldin \cite{DPDRG2}:
\begin{defin}[\cite{DPDRG2}]
We say $F$ satisfies the \emph{atomic condition} if for every probability measure $\mu$ on $G(n,d)$, the following properties hold:
\begin{enumerate}
\item $\dim \mbox{Ker} A_F(\mu) \le n-d$;
\item if $\dim \mbox{Ker} A_F(\mu) = n-d$, then in fact $\mu$ is a Dirac delta on a $d$-plane (i.e., $\mu = \delta_T$ for some $T\in G(n,d)$).
\end{enumerate}
\end{defin}
In intuitive terms, the atomic condition means that whenever you take an average of $B_F$ weighted by $\mu$, the result cannot have the ``full'' kernel dimension $n-d$ unless $\mu$ is supported on a single $d$-plane. Equivalently, as we will see, it means that an $F$-stationary varifold cannot \emph{oscillate between multiple tangent planes} at almost every point.
This condition holds trivially in the special case $F\equiv 1$, i.e., for the area integrand, as in this case $B_F(T)=T$, see the proof sketch of Theorem~\ref{thm:sacvsac} below.

In the general anisotropic case, the atomic condition will serve as our primary \emph{ellipticity assumption} on $F$.
The atomic condition looks difficult to check on specific instances of anisotropic integrands, even numerically. Hence a lot of research has been devoted to relating the atomic condition to more classical ellipticity conditions, as we discuss in the following sections.
\subsection{Convexity notions in the vectorial graphical setting}\label{ellipticitynotions}
We now recall the relevant convexity notions for functions of matrices that arise in the vectorial calculus of variations, and refer to the books \cite{Dac,Mull,Rindler2018} for a thorough description of the topic. Let $f: \mathbb{R}^{(n-d)\times d} \to \mathbb{R}$ be a continuous function on the space of real $(n-d)\times d$ matrices. One should interpret these matrices as gradients of functions $u:\Omega \subset \mathbb{R}^d\to \mathbb{R}^{n-d}$, so that one can look at the functional
$$\mathcal{F}(u)=\int_\Omega f(Du(x))dx, \qquad u:\Omega\subset \mathbb{R}^d\to \mathbb{R}^{n-d},$$
as an anisotropic energy associated to the $d$-dimensional graph of $u$ in $\mathbb{R}^n$.
\begin{defin}
We say $f$ is convex if for all $A,B \in \mathbb{R}^{(n-d)\times d}$ and all $\theta \in [0,1]$,
\[ f(\theta A + (1-\theta)B) \;\le\; \theta\,f(A) + (1-\theta)\,f(B). \]
We say $f$ is rank-one convex if the above inequality holds whenever $\mathrm{rank}(A - B) \le 1$. Equivalently, $f$ is rank-one convex if it is convex along every affine line in $\mathbb{R}^{(n-d)\times d}$ in a direction of rank at most 1.
\end{defin}
It is immediate from the definitions that ``convex $\implies$ rank-one convex'', since rank-one perturbations form a subclass of all perturbations. The notion of rank-one convexity arises naturally in connection with microstructures and laminations in composite materials. In many variational problems involving multiple phases or multiple wells, the absence of rank-one convexity is what allows for fine-scale oscillations leading to microstructure. Rank-one convexity is a necessary condition for quasiconvexity, as we will see.
\begin{defin}
We say $f$ is quasiconvex if for every bounded domain $\Omega \subset \mathbb{R}^d$ (say a unit cube) and every smooth function $\varphi: \Omega \to \mathbb{R}^{n-d}$ compactly supported in $\Omega$, one has
\[ \frac{1}{\mathcal{H}^d(\Omega)} \int_{\Omega} f(A + \nabla \varphi(x))\, dx \;\ge\; f(A), \]
for every constant matrix $A \in \mathbb{R}^{(n-d)\times d}$. If strict inequality holds for all nonzero $\varphi$, we say $f$ is strictly quasiconvex.
\end{defin}
Quasiconvexity, introduced by Morrey \cite{Morrey1, Morrey2}, is an appropriate generalization of convexity in the multidimensional calculus of variations. In fact, it is equivalent to the weak lower semicontinuity of the functional $\mathcal{F}$ in $W^{1,p}$ under standard $p$-growth assumptions, hence providing good existence results for minimizers of $\mathcal{F}$. It is known that
$$\mbox{convex $\implies$ quasiconvex},$$
 by Jensen's inequality applied to $f$.
 Far less trivial is Morrey's result that
 $$\mbox{quasiconvex $\implies$ rank-one convex}.$$
  Intuitively, a rank-one perturbation $A + t(a\otimes b)$ with $a\in \mathbb{R}^{n-d}, b\in \mathbb{R}^d$ can be realized as a gradient perturbation in a fine one-dimensional laminate, so the quasiconvexity inequality along that oscillation forces second derivatives in rank-one directions to be nonnegative.
Quasiconvexity is difficult to verify, and explicit examples are correspondingly scarce \cite{Sve1}. Hence, Ball introduced the notion of polyconvexity as a convenient sufficient condition for existence of minimizers in nonlinear elasticity~\cite{Ball1,Ball2}.
\begin{defin}[Polyconvexity]\label{def:poly}
$f: A\in \mathbb{R}^{(n-d)\times d} \to f(A) \in \mathbb{R}$ is polyconvex if $f$ is a convex function of the minors of $A$.
\end{defin}
Clearly,
$$\mbox{polyconvexity $\implies$ quasiconvexity.}$$
Indeed, denoting with $T(A)$ the vector of all minors of $A$, if $f$ is polyconvex then $f(A) = g(T(A))$ for some $g$ convex. Then by Jensen's inequality, and using that minors are null Lagrangians,
\[ f(A) = g(T(A)) \le \frac{1}{\mathcal{H}^d(\Omega)} \int_\Omega g(T(A+\nabla\varphi)) = \frac{1}{\mathcal{H}^d(\Omega)}\int_\Omega f(A+\nabla\varphi), \]
showing quasiconvexity. 

The idea behind polyconvexity is that even if $f$ is not convex in the matrix entries, it might become convex after ``lifting'' to the space of all minors. For example, consider $d=2, n=4$. If we list the minors of a $2\times 2$ matrix $A$ as $(A, \det A) \in \mathbb{R}^5$, then a function like $f(A) = \det A$ (which is \emph{not} convex in the four entries of $A$) becomes linear (and hence convex) in the extended space. So $f(A)=\det A$ is polyconvex. 

Polyconvexity is strictly stronger than quasiconvexity: for instance, Alibert and Dacorogna \cite{AD} and  Šverák \cite{Sve2} constructed continuous integrands on $2\times 2$ matrices that are quasiconvex but not polyconvex.
In fact, in collaboration with Lei and Young, we have recently shown a sharp characterization: if one requires quasiconvexity to hold not just for single-valued perturbations but for multivalued (Q-valued) perturbations of any order, then that stronger condition is equivalent to polyconvexity \cite{DRLY}. We will elaborate on this in the mini-course.
In summary, we have a hierarchy of implications:
\[ \text{Convex} \;\implies\; \text{Polyconvex} \;\implies\; \text{Quasiconvex} \;\implies\; \text{Rank-one convex}. \]
For $1<d<n-1$, there are known counterexamples to all the reverse implications, with one exception: for $d=2, n=4$, it is a famous open conjecture of Morrey~\cite{Morrey2} whether every rank-one convex integrand must be quasiconvex.
It is important to remark that in codimension one, i.e., when $d=n-1$, all the above convexity conditions coincide.

We conclude recalling the notion of quasiconvexity of $f : \mathbb{R}^{(n-d)\times d} \to \mathbb{R}$ with respect to $Q$-valued Lipschitz graphs, where $Q\in \N$.
To define $Q$-valued functions, let $\delta_{x_i}$ be the Dirac mass in $x_i\in \mathbb R^{n-d}$ and let
\begin{equation*}
\mathcal{A}_Q(\mathbb{R}^{n-d}) :=\left\{\sum_{i=1}^Q\delta_{x_i}\,:\,x_i\in\mathbb R^{n-d}\;\textrm{for every  }i=1,\ldots,Q\right\},
\end{equation*}
equipped with the Wasserstein distance.
A $Q$-valued function is then a map from an open set $U\subset \mathbb R^d$ to $\mathcal{A}_Q(\mathbb{R}^{n-d})$; a Lipschitz $Q$-valued function is a Lipschitz map from $U$ to $\mathcal{A}_Q(\mathbb{R}^{n-d})$. For $\{X_i\}_{i=1}^Q\subset \R^{(n-d)\times d}$, let $\bar{f} : ( \mathbb{R}^{(n-d)\times d})^Q\to  \mathbb{R}$ be defined as
  \begin{equation}\label{integ}
      \bar{f}(X_1, \ldots, X_Q):=\sum_{i=1}^Q f(X_i).
  \end{equation}
  Since $\bar{f}$ is invariant under permutations of $X_i$, one can naturally extend the definition of $\bar{f}$ to $\mathcal{A}_Q(\mathbb{R}^{(n-d)\times d})$.
\begin{defin}\label{qc}
We say that $f$ is quasiconvex for Lipschitz $Q$-valued functions if, for every affine $Q$-valued function
$u (x)= \sum_{j=1}^J Q_j \delta_{a_j + L_j x}$, with $x\in \Omega$, $L_j\in  \mathbb{R}^{(n-d)\times d}$, $a_j\in  \mathbb{R}^{n-d}$, and $Q= \sum_{j=1}^J Q_j$, the following holds.
Given any collection of Lipschitz $Q_j$-valued functions $\varphi^j : \Omega\to \mathcal{A}_{Q_j}(\mathbb{R}^{n-d})$
with $\varphi^j|_{\partial \Omega}(x) = Q_j \delta_{a_j +L_jx}$, we have the inequality
$$
\bar{f}\big(Du\big)\leq  \frac{1}{\mathcal{H}^d(\Omega)}\int_{\Omega} \bar{f}\big(D\varphi^1(x), \ldots, D\varphi^J(x)\big)dx.
$$
\end{defin}
\section{Regularity theory}
\subsection{Rectifiability of $F$-stationary varifolds}
For the area integrand $F \equiv 1$, Allard proved that if a varifold $V$ has locally bounded first variation and has positive density $\|V\|$-almost everywhere, then it must be rectifiable. However, for a general anisotropic integrand $F$, the situation is more subtle: the atomic condition for $F$ is necessary and sufficient to deduce rectifiability of $F$-stationary varifolds. Here we give a brief sketch of the proof of the anisotropic rectifiability theorem that we proved in collaboration with De Philippis and Ghiraldin \cite{DPDRG2}.
We are going to use the following result, which we prove in \cite[Lemma 2.2]{DPDRG2}. This essentially corresponds to \cite[Theorem 4.5]{FraMan}, see also \cite[Theorem 16.7]{Mattila}.
\begin{lemma}[\cite{DPDRG2}] \label{rectlemm}
Let \(\nu\) be a positive Radon measure on $\mathbb{R}^n$ such that:
\begin{itemize}
\item[(i)] For \(\nu\)-a.e. \(x\) there exists \(T_x \in G(n,d)\) such that every \(\sigma\in {\rm Tan}(x,\nu)\)  is translation invariant along \(T_x\).
\item[(ii)] For \(\nu\)-a.e. \(x\), \(0<\Theta^d_*(x,\nu)\le \Theta^{d,*}(x,\nu)<+\infty\).
\end{itemize}
Then \(\nu\) is a \(d\)-rectifiable measure, that is, there exists an $\mathcal{H}^d$-rectifiable set $K$ and a positive Borel function $\theta$ such that \(\nu= \theta \mathcal{H}^d \restr K\), and for \(\nu\)-a.e. $x$, \(T_x\) coincides with the tangent plane of $K$ at $x$.
\end{lemma}
The lemma above is a consequence of a rectifiability criterion due to Preiss \cite[Theorem 5.3]{Preiss}. The following is the main rectifiability theorem proved in \cite{DPDRG2}:
\begin{theorem}[\cite{DPDRG2}]\label{thm:rect}
  \(F\) satisfies the atomic condition if and only if for every $d$-varifold \(V\) whose anisotropic first variation $\delta_FV$ is a Radon measure, the varifold
\begin{equation}\label{Vstar}
V_* :=V\restr (\{x\in \R^n  : \Theta^d_*(x,V)>0\}\times G(n,d))
\end{equation}
  is rectifiable.
\end{theorem}
\begin{proof}[Proof sketch.] Assume for simplicity that $V=V_*$. We divide the proof into two parts.

\textbf{Sufficiency} (atomic condition $\implies$ rectifiability result). One needs to show two key properties for any $V(dx,dT)=\|V\|(dx)\otimes \mu_x(dT)$ whose anisotropic first variation $\delta_FV$ is a Radon measure:
\begin{itemize}
\item[(a)] For $\|V\|$-a.e. $x$ there exists \(T_x \in G(n,d)\) such that $\mu_x=\delta_{T_x}$;
\item[(b)]  \(\|V\|\ll \mathcal{H}^d\).
\end{itemize}
Once (a) and (b) are established, one can invoke Lemma \ref{rectlemm} to conclude.
We first give the sketch of the proof of (a).
For \(\|V\|\)-a.e. $x$, every \(W\in {\rm Tan} (x, V)\) can be written as \(W( dy, dT)=\sigma(dy)\otimes \mu_x(dT)\) for some \(\sigma\in {\rm Tan} (x,\|V\|)\) and
$$
 0= \delta_{F} W(g) = \int_{B(0,1)} A_F(\mu_x) : Dg(y)  \, d\sigma (y), \quad \forall g \in  C_c^1(B(0,1),\R^n).
$$
Therefore  \(A_F(\mu_x)D \sigma=0\) in the sense of distributions.
Since \( \mbox{Ker} A_F(\mu_x)=(\mathrm{Im} A_F(\mu_x)^T)^\perp\), we get:
$$
\partial_e \sigma=0 \qquad \mbox{for all  $e\in T_x:=\mathrm{Im} A_F(\mu_x)^T$.}
$$
Then we conclude with the following three observations:
 \begin{itemize}
 \item atomic condition \quad $\Rightarrow$ \quad $\dim(T_x)\geq d$;
 \item positive lower density at $x$ \quad $\Rightarrow$ \quad $\dim(T_x)= d$;
 \item atomic condition \quad $\Rightarrow$ \quad $\mu_x=\delta_{T_x}$.
 \end{itemize}
To prove property (b), we prove the following lemma, inspired by the strong constancy lemma of Allard \cite{All4}. We consider the Radon-Nikodym decomposition $\|V\|=f\,\mathcal H^d \restr  \{\Theta^d_*(\cdot,V)>0\}+\|V\|^s$, where $\|V\|^s$ denotes the singular part of $\|V\|$ with respect to the measure $\mathcal H^d \restr  \{\Theta^d_*(\cdot,V)>0\}$, and we apply the lemma at a density point $x$ of $\|V\|^s$ such that
$$V_j:=V_{x,r_j}\rightharpoonup \sigma \otimes \delta_{T_x} \mbox{ in $B(0,1)$}.$$
\begin{lemma}[\cite{DPDRG2}]
If
$$\sup_j|\delta_{F} V_j |(B(0,1)) <\infty , \quad \mbox{ and } \quad  \lim_{j\to \infty}\int_{B(0,1)\times G(n,d)} |T-T_x|  dV_j(z,T) =0,$$
then there exists $\gamma \in L^1(B^d(0,1))$ (with respect to $\mathcal{H}^d$) such that for every $0<t<1$
$$
\Big| (\pi_{T_x})_{\#}\|V_j\| - \gamma\mathcal{H}^d\restr B^d(0,1)\Big|(B^d(0,t) ) \longrightarrow 0.
$$
\end{lemma}
\textbf{Necessity} (rectifiability result $\implies$ atomic condition). If $F$ does not enjoy the atomic condition, there exists some probability measure $\mu$ on $G(n,d)$ such that:
\begin{itemize}
\item either $\dim \mbox{Ker } A_F(\mu) > n-d$,
\item or $\dim \mbox{Ker } A_F(\mu) = n-d$ but $\mu \neq \delta_T$ for every $T\in G(n,d)$.
\end{itemize}
In either case, one can leverage $\mu$ to construct a $d$-varifold that is $F$-stationary and with positive density, yet that is not rectifiable, in the following way.
Define $W:=\mathrm{Im} A_F(\mu)^T$,  $k:=\dim W \leq d$ and
$$
V:=\mathcal{H}^k \restr W \otimes \mu.
$$
\(V\) is not \(d\)-rectifiable: if \(k<d\), then $\|V\|$ is supported on the $k$-dimensional plane $W$; if \(k=d\), then \(W\in G(n,d)\) and \(\mu\ne \delta_{W}\).
However, for  every  \(x\in W\) the density can be checked to be positive:
\begin{equation*}
\Theta^d(x,V)=\lim_{r \to 0} \frac{\mathcal{H}^k(B(x,r)\cap W)}{\omega_d r^d}=
\begin{cases}
1 \qquad &\textrm{if \(k=d\)}\\
+\infty &\textrm{if \(k<d\)}
\end{cases}>0.
\end{equation*}
Moreover $\delta_F V = 0$. Indeed, for every $g \in  C_c^1(\mathbb{R}^n,\mathbb{R}^n)$
$$
  \delta_F V(g) = \int_{W} A_F(\mu) : Dg  \, d\mathcal{H}^k  = \int_{W} \operatorname{tr}\big(A_F(\mu)^T \, Dg \, P_W\big)  \, d\mathcal{H}^k =0,
  $$
where $P_W$ denotes the orthogonal projection onto $W$. In the second equality we used that $A_F(\mu)^T=P_W A_F(\mu)^T$, since $\mathrm{Im}\, A_F(\mu)^T= W$, and in the last equality we used that $\int_W \partial_e \varphi \, d\mathcal{H}^k=0$ for every $\varphi\in C^1_c(\R^n)$ and every $e\in W$, as such derivatives are tangential to the plane $W$.
This constructs a counterexample to the validity of the rectifiability result and concludes the proof.
\end{proof}
One of the major applications of Theorem~\ref{thm:rect} is in developing an anisotropic min-max theory, as we will discuss in Section \ref{ch:mm}. However it also has applications to other settings, such as the anisotropic isoperimetric and capillary problems \cite{DRN,DNR,DRKS}.
\subsection{Atomic condition versus standard ellipticity}
Although the atomic condition is the natural ellipticity condition to assume on an anisotropic integrand, since it is equivalent to the rectifiability theorem, it may be difficult to check it on specific instances of integrands. The aim of this section is to remove this obstruction by comparing the atomic condition with classical notions of ellipticity.
\subsubsection{Codimension one}\label{cod1conv}
The following theorem, proved in collaboration with De Philippis and Ghiraldin \cite{DPDRG2}, completely settles the understanding of the atomic condition in codimension one.
\begin{theorem}[\cite{DPDRG2}]
If $d=n-1$, the atomic condition is equivalent to strict convexity in spherical directions of the 1-homogeneous extension of the function $G:\mathbb S^{n-1}\to \mathbb R$, defined as \(G(\nu):=F( \nu^{\perp})\), i.e.,
$$G(\nu)>\langle d_\nu G(\bar \nu),\nu\rangle, \quad \forall \nu\neq \pm \bar \nu \in \mathbb S^{n-1}.$$
\end{theorem}
\begin{proof}[Proof idea.]
 When $d=n-1$
$$
A_F(\mu)= \int_{ \mathbb S^{n-1}}\Big(G(\nu)\mbox{Id}-\nu \otimes d_\nu G(\nu)\Big)d\mu.
$$
If there exists \(\bar \nu\in \mbox{Ker }  A_F(\mu)\cap \mathbb S^{n-1}\), we compute:
\[
0=\langle d_\nu G(\bar \nu),A_F(\mu)\bar\nu \rangle=\int_{\mathbb S^{n-1}} G(\bar \nu)G(\nu)-\langle d_\nu G(\bar \nu),\nu\rangle\langle d_\nu G( \nu),\bar \nu\rangle d\mu(\nu) \geq 0.
\]
From this, it is not difficult to conclude that $\mu=\frac12(\delta_{-\bar \nu}+\delta_{\bar \nu})$, where with an abuse of notation we have lifted the measure $\mu$ to an even measure on $\mathbb S^{n-1}$.
\end{proof}
\subsubsection{Arbitrary codimension: necessary conditions for the validity of the atomic condition}
In arbitrary codimension, the situation is more complex, as there are many ellipticity conditions, as already discussed in Section \ref{ellipticitynotions} for the graphical setting.
The first result we obtained in arbitrary codimension in collaboration with Kolasiński \cite{DRK} is the following:
\begin{theorem}[\cite{DRK}]\label{almell}
The atomic condition implies Almgren's strict ellipticity, that is:
        \begin{displaymath}
          \mathcal{H}^d(S) > \mathcal{H}^d(D) \quad \Rightarrow  \quad \mathbf{F}(S) > \mathbf{F}(D)  \,,
        \end{displaymath}
        for every $d$-cube $D$ and every compact $d$-rectifiable set $S$ which cannot be retracted onto $\partial D$ via Lipschitz deformations.
\end{theorem}
Almgren's strict ellipticity is a geometric version of the strict quasiconvexity. Roughly speaking, among all surfaces spanning the boundary of a flat $d$-cube, the flat cube is the unique $F$-minimizer (not only the unique area-minimizer). Theorem~\ref{almell} says the atomic condition implies that no other  minimizer exists for a flat boundary, even in a broader topological sense, while in the classical quasiconvexity only graphs of Lipschitz functions are allowed as competitors.
\begin{proof}[Proof idea.]
The proof relies on the following homogenization argument. If Almgren's strict ellipticity is not satisfied, up to solving a set-theoretic anisotropic Plateau problem, there exists a $d$-rectifiable competitor set $S$ which cannot be retracted onto $\partial D$ via Lipschitz deformations, with $\mathcal{H}^d(S) > \mathcal{H}^d(D)$, but $\mathbf{F}(S)\leq \mathbf{F}(D)$. For every $i\in \mathbb N$, we construct a tiling $S_i$ of $i^d$ copies of the $\frac{1}{i}$-rescaling of $S$ with boundaries contained in the $(d-1)$-skeleton of a grid covering $D$. We can prove that $\delta_{F} S_i = 0$ on $\R^n\setminus \partial D$, by showing via algebraic topology arguments that $S_i$ cannot be retracted onto $\partial D$ via Lipschitz deformations. It follows that as a varifold $S_i\rightharpoonup W:=\theta \mathcal H^d\restr D\otimes \mu$, where $\theta >1$ as $\mathcal{H}^d(S_i)=\mathcal{H}^d(S) > \mathcal{H}^d(D)$ is constant in $i$, and $\delta_{F} W = 0$ on $\R^n\setminus \partial D$. However $W$ is not rectifiable, as otherwise we would have $\mu=\delta_{P_0}$ with $P_0=\operatorname{span}(D)$, which yields the following contradiction:
$$\mathbf{F}(D)<\mathbf{F}(W)=\mathbf{F}(S_i)=\mathbf{F}(S)\leq \mathbf{F}(D).$$
Hence $F$ does not satisfy the atomic condition.
\end{proof}
The strategy of this proof is quite robust. In particular we can potentially prove that the atomic condition implies other ellipticity conditions, provided that violating these conditions allows one to build fillings of the flat boundary with better or equal $F$-area than the flat filling. 

Building on this approach, in collaboration with Lei and Young \cite{DRLY}, we proved that the atomic condition ensures the strict polyconvexity of the integrand when suitably extended to oriented $d$-planes. For this aim we will need some extra notation.
Let $\widetilde{G}(n,d)$ denote the oriented Grassmannian, and let $p: \widetilde{G}(n,d) \to G(n,d)$ denote the forgetful projection. For every $P\in \widetilde{G}(n,d)$ we denote by $\omega_P$ the unit simple $d$-vector associated to $P$. For a rectifiable $d$-current $\Sigma=(E, \theta, P)$, the Gaussian image is a Radon measure on $\widetilde{G}(n,d)$ defined as
$$\gamma_\Sigma:=P_\#(\theta \mathcal{H}^d \restr E).$$
\begin{theorem}[\cite{DRLY}]\label{thm:DRLY}
The atomic condition implies strict polyconvexity of the even extension $\tilde F$ of $F$ to $\widetilde{G}(n,d)$, that is for every $P_0\in \widetilde{G}(n,d)$ and every positive Radon measure $\mu$ on $\widetilde{G}(n,d)$ satisfying $\int \omega_P d \mu(P) = \omega_{P_0}$,
 \begin{equation}\label{poly}
 \int \tilde F(P) d\mu(P)\geq \tilde F(P_0)
 \end{equation}
    with equality if and only if $\mu=\delta_{P_0}$.
\end{theorem}
We observe that Jakimiuk, Kolasiński and Leśniak have recently shown in \cite{JKL} that the converse implication of Theorem \ref{thm:DRLY} is false. Indeed, for all $d \geq 2$ and $n-d\geq 3$, they construct strictly polyconvex integrands, which fail the atomic condition.

We first justify the name polyconvexity, comparing \eqref{poly} with Definition \ref{def:poly}. For a matrix $X\in \mathbb R^{(n-d)\times d}$, let
  $$M(X):=  \begin{pmatrix} \operatorname{Id}_d\\ X \end{pmatrix}\in \mathbb R^{n\times d}$$
  and $\bigwedge M(X) = w_1(X)\wedge\dots\wedge w_d(X)$, where $w_i(X)$ are the columns of $M(X)$.
 The coordinates of $\bigwedge M(X)$ in the standard basis of $\bigwedge^d\mathbb{R}^n$ are the $d\times d$ minors of $M(X)$, which in turn coincide with the minors of $X$ of any order. Hence $f:\R^{(n-d)\times d}\to (0,\infty)$ is polyconvex in the sense of Definition \ref{def:poly} if and only if it is a convex function in the space of coordinates of $\bigwedge M(X)$ in the standard basis of $\bigwedge^d\mathbb{R}^n$. This is exactly \eqref{poly}.

\begin{proof}[Proof sketch.]
 Suppose $\tilde F$ is not strictly polyconvex. Then there are $P_0 \in \widetilde{G}(n,d)$ and a positive Radon measure $\mu$ on $\widetilde{G}(n,d)$ such that $\int \omega_P d \mu(P) = \omega_{P_0}$, $\mu\neq\delta_{P_0}$, and
  $$\int \tilde F(P) d \mu(P)=\inf_{\substack{\nu\in \mathcal{M}(\widetilde{G}(n,d))\\ \int \omega_P d \nu(P) = \omega_{P_0}}} \left\{\int \tilde F(P)d \nu(P) \right\}\leq \tilde F(P_0).$$
If we are able to construct a sequence of rectifiable currents $\{\Sigma_N\}_{N\in \mathbb{N}}$ such that
$$\partial \Sigma_N=\partial D, \qquad \gamma_{\Sigma_N} \rightharpoonup^* \mu,  \qquad  V_{\Sigma_N} \rightharpoonup^* W:= (\mathcal H^d\restr D)\otimes p_\#\mu,$$
which is guaranteed by Theorem~\ref{polapprox} below,
then $\{\Sigma_N\}_{N\in \mathbb{N}}$ is a minimizing sequence among rectifiable currents with boundary $\partial D$. Hence $W$ is stationary away from $\partial D$ and not rectifiable.
 We conclude $F$ does not satisfy the atomic condition.
\end{proof}
Hence we are left to show the following:
\begin{theorem}[\cite{DRLY}]\label{polapprox}
 Let $P_0 \in \widetilde{G}(n,d)$, $\mu$ be a positive Radon measure on $\widetilde{G}(n,d)$, $D=[0,1]^d\subset P_0$, and $E=[0,1]^n$.
 \begin{itemize}
\item If $\int \omega_P d \mu(P) = 0$, there exist polyhedral $d$-chains $\{\Sigma_N\}_{N\in \mathbb{N}}$ such that
 $$\partial \Sigma_N = 0, \; \, \gamma_{\Sigma_N} \rightharpoonup^*\mu, \;\, \supp \Sigma_N \xrightarrow{\mathrm{Haus}} E, \;\,  V_{\Sigma_N}\rightharpoonup (\mathcal{H}^n\restr E)\otimes p_\#\mu.$$
\item If $\int \omega_P d \mu(P) = \omega_{P_0}$, there exist polyhedral $d$-chains $\{\Sigma_N\}_{N\in \N}$ such that
  $$ \partial \Sigma_N =\partial [D], \; \, \gamma_{\Sigma_N} \rightharpoonup^*\mu, \;\,  \supp \Sigma_N \xrightarrow{\mathrm{Haus}} D, \;\,  V_{\Sigma_N}\rightharpoonup (\mathcal{H}^d\restr D)\otimes p_\#\mu.$$
  \end{itemize}
Moreover, if $\supp \mu \subset  \widetilde{G}^+(n,d)$ consists of positively oriented $d$-planes with respect to $P_0$, then we can construct $\Sigma_N$ as Lipschitz multigraphs over $P_0$.
\end{theorem}
Burago and Ivanov \cite{BurIva} gave a nonconstructive proof of the second bullet point of Theorem~\ref{polapprox}. In \cite{DRLY} we prove Theorem~\ref{polapprox} by giving an explicit construction of such polyhedral $d$-chains. This proof is robust enough to generalize to other classes of surfaces. For instance, as stated in Theorem~\ref{polapprox}, in the case the tangent planes are assumed to be positively oriented $d$-planes, we can construct Lipschitz multigraphs.
\begin{proof}[Proof sketch.]
We start with the case $\int \omega_P d \mu(P) = 0$.
By a density argument, assume
$$\mu := \sum^{k}_{i=1} m_i \delta_{P_i},$$
 where $P_i$ are planes of rational slope, i.e., there exists $M_i\in M_{n-d,n}(\Z)$ of rank $n-d$ such that $P_i=\mbox{Ker } M_i$.

The linear map $M_i: \R^n\to \R^{n-d}$ descends to $\hat{M}_i: \R^n/\Z^n\to \R^{n-d}/\Z^{n-d}$, and we denote $\hat\Sigma_i=\hat{M}_i^{-1}(0)\subset \R^n/\Z^n$. We define the polyhedral chains in $\R^n/\Z^n$:
$$\hat T_i:=\frac{[\hat\Sigma_i]}{{\mathbb{M}}([\hat\Sigma_i])}, \qquad \hat S:=\sum_{i=1}^k m_i \hat T_{i}.$$
We have that $\partial \hat S=0$ and $\gamma_{\hat S} = \mu$. Also we observe that \(\widehat S\) is null-homologous: under the canonical identification
\[
H_d(\mathbb R^n/\mathbb Z^n;\mathbb R)\simeq \bigwedge\nolimits^d\mathbb R^n,
\]
the normalized current
\(\widehat T_i\) represents the unit simple \(d\)-vector \(\omega_{P_i}\). Indeed, for every constant \(d\)-form \(\alpha\) on \(\mathbb R^n/\mathbb Z^n\),
\[
\widehat T_i(\alpha)
=
\frac{1}{\mathbb M([\widehat\Sigma_i])}
\int_{\widehat\Sigma_i}
\langle\alpha,\omega_{P_i}\rangle d\mathcal H^d
=
\langle\alpha,\omega_{P_i}\rangle.
\]
It follows that
\[
\widehat S(\alpha)=
\sum_{i=1}^k m_i\widehat T_i(\alpha)=
\left\langle
\alpha,\sum_{i=1}^k m_i\omega_{P_i}
\right\rangle
=\left\langle
\alpha,\int\omega_P d\mu(P)
\right\rangle
=0.
\]
Since the de Rham cohomology of the torus is generated by constant forms, we deduce that
\(\widehat S=0\) in $H_d(\mathbb T^n;\mathbb R)$.
Hence there is a real $(d+1)$-polyhedral chain $\hat Q$ in $\R^n/\Z^n$ such that $\partial \hat Q = \hat S$.  Let $\tilde{Q}$, $\tilde{S}$ be the lifted polyhedral chains of $\hat Q$ and $\hat S$ to $\R^n$.
We consider a building block for the tiling defined as:
$$
S := \partial  (\tilde{Q}\restr E) = \partial \tilde{Q}\restr E + \tilde{Q}\cap \partial E = \tilde{S}\restr E + \tilde{Q}\cap \partial E.
$$
We observe that $\partial S = \partial^2(\tilde{Q}\restr E)=0$, and $\gamma_{\tilde{S}\restr E}=\mu$.
We construct $\Sigma_N$ by tiling $N^n$ rescaled copies of $S$ in subcubes of $[0,1]^n$ of size $1/N$, and we conclude the following desired properties as $N\to \infty$:
\begin{itemize}
\item $\partial \Sigma_N=0$.
\item $\gamma_{\Sigma_N}\rightharpoonup^* \mu$ since the terms arising from the tiling of $\tilde{Q}\cap \partial E$ (with no control on the Gaussian image) cancel out in $(0,1)^n$ and survive only on $\partial E$.
\item $\supp \Sigma_N \to E$ in Hausdorff distance by the tiling procedure.
\item $V_{\Sigma_N}\rightharpoonup (\mathcal{H}^n\restr E)\otimes p_\#\mu$ by homogenization.
\end{itemize}
The proof of the case $\int \omega_P d \mu(P) = \omega_{P_0}$ is obtained by replacing $\hat S$ in the previous construction with:
$$\hat S := \sum^k_{i=1}m_i \hat T_{i} + \hat  T_{-P_0}.$$
In this case we have to tile $(N^2)^d \times N^{n-d}$ rescaled copies of $S$ in subcubes of $[0,1]^d\times [0,1/N]^{n-d}$ of size $1/N^2$, and we remove all $d$-cubes generated by $\hat T_{-P_0}$ and correct the boundary with annuli contained in $(\partial [0,1]^d)\times [0,1/N]^{n-d}$ of negligible mass $\lesssim 1/N$.

To conclude, constructing Lipschitz multigraphs when $\supp \mu \subset  \widetilde{G}^+(n,d)$ requires correcting all vertical and negatively oriented error terms.
\end{proof}
Theorem~\ref{polapprox} provides the following surprising consequences for the convexity notions of $f: \mathbb{R}^{(n-d)\times d} \to \mathbb{R}$:
\begin{theorem}[\cite{DRLY}]\label{thm:int}
$f: \mathbb{R}^{(n-d)\times d} \to \mathbb{R}$ is polyconvex if and only if $f$ is quasiconvex for Lipschitz $Q$-valued functions for every $Q\in \mathbb N$.
\end{theorem}
\begin{cor}[\cite{DRLY}]
  There exists an $f: \mathbb{R}^{(n-d)\times d} \to \mathbb{R}$ that is quasiconvex for Lipschitz $1$-valued functions but that is not quasiconvex for Lipschitz $Q$-valued functions for some $Q\in \mathbb N$.
\end{cor}
We observe that the use of Lipschitz multigraphs in Theorem~\ref{polapprox} is essential: in general, they cannot be replaced by single-valued Lipschitz graphs. Indeed, suppose that the conclusion of Theorem~\ref{polapprox} held with single-valued graphs for every positive measure
\(\mu\) on \(\widetilde G^{+}(n,d)\) satisfying
\[
\int_{\widetilde G(n,d)}\omega_P d\mu(P)=\omega_{P_0}.
\]
Then the argument used to prove Theorem~\ref{thm:int}, applied with \(Q=1\), would imply that every quasiconvex integrand
\[
f:\mathbb R^{(n-d)\times d}\longrightarrow \mathbb R
\]
is polyconvex. This is false already for \(2\times2\) matrices: there exist continuous quasiconvex integrands which are not polyconvex, by the counterexamples of Alibert-Dacorogna and Šverák \cite{AD,Sve2}.

Actually, the approximation with Lipschitz multigraphs in Theorem~\ref{polapprox} is sharp, that is, every $Q \in \mathbb N$ is needed for the validity of the theorem. Indeed, in collaboration with De Gennaro \cite{DGDR},  for any fixed \(Q \in \mathbb{N}\) we construct a positive measure on $\widetilde{G}^+(4,2)$, whose barycenter is a simple \(2\)-vector, yet which cannot be approximated by weighted Gaussian images of Lipschitz \(Q\)-graphs. This construction extends to $\widetilde{G}^+(n,d)$ whenever \(n-2 \ge d\geq 2\).
As an application, we prove that for every $Q\geq 1$ and $p\ge 2$ there exists a non-polyconvex $Q$-integrand, in the sense introduced by De Lellis, Focardi and Spadaro \cite{DLFS}, whose associated energy is weakly lower semicontinuous in $W^{1,p}$. These $Q$-integrands suggest that the validity of Theorem~\ref{thm:int} probably also requires all $Q \in \mathbb N$. However, the latter remains an open question, see \cite[Remark~1.14]{DRLY}.
\subsubsection{Arbitrary codimension: sufficient conditions for the validity of the atomic condition}
In joint work with Tione \cite{DRT}, we introduced the following sufficient condition for the atomic condition:
\begin{defin}[\cite{DRT}]
Denote $F^*:G(n,n-d)\to (0,\infty)$, $F^*(S):=F(S^\perp)$.
\vspace{0.1cm}

$F$ satisfies the \emph{scalar atomic condition} if
\[
B_F(T) : B_{F^*}(S^\perp) > 0, \quad\forall T \neq S \in G(n,d).
\]
$F$ satisfies the \emph{uniform scalar atomic condition} if there exists $C>0$ such that
\[
B_F(T) : B_{F^*}(S^\perp) \ge C\|T-S\|^2, \quad\forall T,S \in G(n,d).
\]
\end{defin}
\begin{theorem}[\cite{DRT}]\label{thm:sacvsac}
If an integrand $F$ satisfies the scalar atomic condition, then it satisfies the atomic condition.
\end{theorem}
\begin{proof}[Proof sketch.]
For simplicity, we give the proof for $F\equiv 1$, the general proof being very similar.
For $F\equiv 1$,
\[
B_F(T)=T \qquad  B_{F^*}(S^\perp)=S^\perp.
\]
Let $\mu \in \mathcal P (G(n,d))$ satisfy  $$\mbox{dim Ker } A_F(\mu)\geq n-d.$$ Choose
$$S^\perp\subset  \mbox{Ker } A_F(\mu), \qquad \mbox{dim } S^\perp = n-d.$$
Hence
$$
0=A_F(\mu): S^\perp=\int_{G(n,d)} T:S^\perp d\mu(T)=\frac12 \int_{G(n,d)}   \|T-S\|^2 d\mu(T).
$$
Then $T=S$ $\mu$-a.e., hence $F$ satisfies the atomic condition.
\end{proof}
The advantage of the scalar atomic condition is that it is easier to check for an integrand $F$, at least numerically, as it requires the positivity of a single function.

The uniform scalar atomic condition, which is a natural uniform quadratic strengthening of the scalar atomic condition, turned out to be useful in arbitrary codimension for both providing examples of non-trivial energies satisfying the atomic condition, and establishing $C^{1,\alpha}$-partial regularity for $F$-stationary Lipschitz graphs.
\begin{theorem}[\cite{DRT}]
The uniform scalar atomic condition is an open condition in the $C^2$ topology, i.e., if $F$ satisfies the uniform scalar atomic condition, there exists $\varepsilon(F,n,d)>0$ such that every $F'\in C^2(G(n,d); (0,\infty))$ satisfying
\[
\|F - F'\|_{C^2(G(n,d))} \le \varepsilon
\]
also satisfies the uniform scalar atomic condition.
\end{theorem}
In particular, since $F\equiv 1$ satisfies the uniform scalar atomic condition, a $C^2$-neighborhood of the area integrand satisfies the atomic condition. These are the first non-trivial examples in arbitrary codimension.
\begin{theorem}[\cite{DRT}]
Let $F\in C^2$ satisfy the uniform scalar atomic condition, and $\Omega\subset \R^d$ open, bounded. Let $u \in \Lip(\Omega;\R^{n-d})$ be such that its graph has $F$-mean curvature $H$ in $L^p$, $p>d$. Then there exists $\alpha>0$ such that
\[
u \in C^{1,\alpha}(\Omega_0,\R^{n-d}), \quad \mbox{with } \Omega_0\subset \Omega \mbox{ open}, \quad \H^d(\Omega_0)= \H^d(\Omega).
\]
\end{theorem}
In particular, regularity holds in a sufficiently small $C^2$ neighborhood of the area functional.
We remark that in collaboration with Resende \cite{DRR}, we extended the interior regularity theorem above to boundary points with density ratio less than $1/2+\varepsilon$.
\begin{proof}[Proof sketch.]
In the first variation of the graph $\Gamma$ of $u$:
\begin{equation*}
\int B_F(T_y\Gamma):Dg(y) d\|\Gamma\|(y) = -\int \langle H(y),g(y)\rangle d\|\Gamma\|(y),
\end{equation*}
we plug the vector field
\[
g(y):= \varphi^2(y) B_{F^*}(S^\perp) y, \qquad \varphi \mbox{ is a cutoff function supported in $B(0,2r)$ and equal to one on $B(0,r)$,}
\]
to get a Caccioppoli inequality
\begin{equation*}
\fint_{B(0,r)}\|T_y\Gamma - S\|^2\,d\|\Gamma\|(y) \le \frac{C_2}{r^2}\fint_{B(0,2r)}\dist(y,S)^2\,d\|\Gamma\|(y) + r^2C_2\fint_{B(0,2r)}\|H(y)\|^2\,d\|\Gamma\|(y).
\end{equation*}
This reads as the following Caccioppoli inequality for $u$. There exist constants $C(n,d,F,\|u\|_{\Lip}) > 0$ and $k(\|u\|_{\Lip})>0$ such that
\begin{equation*}
\begin{split}
\fint_{B^d(x_0,r)}\|Du(x) - A\|^2\,dx \le &\frac{C}{r^2}\fint_{B^d(x_0,kr)}\|u(x)-(u)_{x_0,kr} - A(x-x_0)\|^2\,dx\\
& + r^2C\fint_{B^d(x_0,kr)}\|H(x,u(x))\|^2\,dx
\end{split}
\end{equation*}
for all $A \in \R^{(n-d)\times d}$, $x_0 \in \Omega$, and $r > 0$.
A contradiction blow-up argument and the regularity theory for the linearized problem, combined with the Caccioppoli inequality, provide decay for the excess:
$$E(x,r) := \fint_{B^d(x,r)}\|Du(y) - (Du)_{x,r}\|^2 dy.$$
In particular, setting $H'(x):=H(x,u(x))$, there exists $c > 0$ such that for every $\tau \in \left(0,\frac{1}{4k}\right)$, there exists $\varepsilon = \varepsilon(\tau)>0$ such that the following bounds
\[
E(x,r) \le \varepsilon(\tau) \quad \text{ and } \quad r^{1-\frac dp}\|H'\|_{L^p(\Omega)} \le E(x,r)
\]
imply
$$
E(x,\tau r) \le c\tau^2E(x,r).
$$
We deduce that at a Lebesgue point $x$ of $Du$, we have $E(x,r)\leq cr^{2\alpha}$. This yields H\"older continuity of $Du$ near $x$.
\end{proof}
This regularity theorem for graphs cannot be directly extended to stationary
varifolds around flat points of density one, due to the lack of a lower density
bound. An anisotropic analogue of the Michael-Simon inequality \cite{MS} would
be decisive in this direction, but it is currently available only for $d=2$,
$n=3$, and for anisotropic integrands that are $C^1$-close to the area
integrand, as proved by De Philippis and Pigati \cite{DPP}. In codimension one,
the $C^{1,\alpha}$-regularity for varifolds is better understood \cite{All4,KS}.
We conclude this section by referring the reader to the recent work \cite{JKL}
for a reinterpretation of the atomic condition in terms of convex geometry and
for quantitative versions of the atomic condition.
\section{Min-max theory for anisotropic minimal hypersurfaces}\label{ch:mm}
\subsection{Introduction and main results}
 In this third lecture, using a min-max variational method we construct closed anisotropic minimal hypersurfaces in closed Riemannian manifolds, with optimal regularity properties. By optimal regularity, we mean that the size of the singular set of the constructed hypersurface matches what is known for solutions of the anisotropic Plateau problem in codimension one \cite{SSA}. In particular, our work in collaboration with De Philippis and Li \cite{DPDR,DPDRL} ensures that the $\mathcal{H}^{n-3}$-measure of the singular set is zero, where $n$ denotes throughout this section the dimension of the ambient manifold.
 
There is a vast literature about the min-max theory in the isotropic setting, $F\equiv 1$. Without attempting to give a full literature review of the isotropic min-max theory, we just recall that already Birkhoff \cite{Birk} in 1917 proved the existence of closed geodesics in manifolds that are homeomorphic to the sphere. However, for higher dimensions we had to wait until 1981, with the work of Almgren and Pitts \cite{Alm2,Pitts}, who proved the existence of closed smooth minimal hypersurfaces in closed Riemannian manifolds, for $3\leq n\leq 6$. This result was one of the main motivations that led Almgren to introduce the theory of varifolds \cite{Alm1}. Also, it relied on the curvature estimates for stable hypersurfaces proved by Schoen, Simon and Yau \cite{SSY}. Schoen and Simon \cite{SS} then extended the existence of closed minimal hypersurfaces with a singular set of Hausdorff dimension at most $n-8$, for every $n\geq 3$.

Allard in 1983~\cite[page 288]{All2} conjectured the validity of a similar construction in the anisotropic setting. However, he observed that ``there remains a considerable amount of work to do before this becomes feasible for general integrands''. The aim of this last section is to address this conjecture.
For the remainder of the mini-course,  we fix an $n$-dimensional smooth Riemannian manifold $M$ without boundary and we denote by
 $$G_{n-1}(M):=\{(x,S):x\in M,\ S\in G(T_xM,n-1)\}.$$
 The anisotropic integrands that we consider from now on are functions
 $$F \in C^3(G_{n-1}(M); (0,\infty))$$
 that are uniformly convex in the second variable, in the sense explained in Section \ref{cod1conv}. In particular these functions $F$ satisfy the atomic condition.
The results presented here are based on the two works \cite{DPDR,DPDRL}. In the first paper \cite{DPDR}, in collaboration with De Philippis, we handled the $3$-dimensional case ($n=3$) and proved the existence of a closed anisotropic minimal surface with \emph{at most one singular point}:
\begin{theorem}[\cite{DPDR}]\label{1}
If $n=3$, there is a nontrivial anisotropic minimal surface $\Sigma\subset M$ without boundary and whose singular set contains at most one point $p\in M$, i.e., $\Sigma$ is $C^2$ embedded away from $p$.
\end{theorem}
In the second work \cite{DPDRL}, in collaboration with De Philippis and Li, we achieved the optimal result in every dimension:
\begin{theorem}[\cite{DPDRL}]\label{2}
For every $n\geq 3$, there exists a smooth embedded anisotropic minimal hypersurface $\Sigma \subset M$ such that $\text{Sing}(\Sigma) \equiv \overline{\Sigma}\setminus \Sigma$ has $\mathcal{H}^{n-3}(\text{Sing}(\Sigma)) = 0$.
 \end{theorem}
Theorem~\ref{2} in particular improves the earlier Theorem~\ref{1}, because for $n=3$, we obtain $\mathcal{H}^0(\text{Sing}(\Sigma))=0$, which in turn implies there are no singular points at all. Hence, the anisotropic minimal surface we construct is actually smooth. We will later discuss the new ideas that were needed to eliminate that one singular point for $n=3$ and to handle the general higher-dimensional case.

We remark that an alternative strategy for constructing closed (isotropic) minimal hypersurfaces in closed Riemannian manifolds is provided by Allen-Cahn approximations \cite{Ilmanen,HutchinsonTonegawa2000,Tonegawa2005,TonegawaWickramasekera2012,Wickramasekera2014,Guaraco2018}. In joint work with Pigati, we have initiated a program to construct anisotropic min-max minimal hypersurfaces via a suitable anisotropic Allen-Cahn approximation \cite{DRP}. Specifically, we study an anisotropic analogue of the Allen-Cahn energy and establish a Modica-type gradient estimate for its critical points. Combining this with curvature estimates for stable solutions, we show that the energy densities of stable or bounded Morse index critical points concentrate along an integer rectifiable varifold that is stationary with respect to the anisotropic integrand. As a consequence, we obtain a (possibly singular) anisotropic min-max hypersurface via the Allen-Cahn approach, yielding an anisotropic analogue of the classical result of Hutchinson and Tonegawa \cite{HutchinsonTonegawa2000}.

To conclude, we recall that, in order to construct (isotropic) minimal hypersurfaces with control on the genus, a natural refinement of the min-max procedure is to seek critical points within a fixed \emph{isotopy class}. This restriction introduces significant difficulties in the regularity theory, since the class of admissible competitors is much smaller. Moreover, the genus need not be lower semicontinuous along sequences of surfaces, so a fine control of the limit is required.
A major breakthrough in this (isotropic) direction was achieved in $3$-dimensional ambient manifolds by Meeks, Simon and Yau \cite{MSY}, who showed that minimizing sequences within a fixed isotopy class converge, in the sense of varifolds, to smooth limits and enjoy lower semicontinuity of the genus. Their approach relies on the regularity theory developed by Almgren and Simon \cite{AS} and has had important topological applications. Building upon the regularity framework of \cite{MSY,AS}, Simon and Smith \cite{smith} outlined a min-max theory for surfaces in $3$-manifolds constrained to isotopy classes. This program was completed by Colding and De Lellis \cite{CDL} and De Lellis and Pellandini \cite{DP}, and later refined by Ketover \cite{Ketover}, who established sharp genus bounds for Heegaard sweep-outs.

In recent joint work with Halavati and Wang \cite{DRHW}, we establish an anisotropic analogue of the celebrated Meeks-Simon-Yau theorem: every minimizing sequence of surfaces within a fixed isotopy class converges to a smooth stable anisotropic minimal surface, and the genus is lower semicontinuous along the sequence. This result also strengthens White's foundational existence theory for anisotropic minimal disks \cite{W3}: for a fixed boundary curve on the $2$-sphere, White proved that \emph{at least one} minimizing sequence of smooth disks converges to a smooth disk, while in \cite{DRHW} we show that \emph{every} minimizing sequence of disks enjoys this property.
As an application, we develop an anisotropic Simon-Smith min-max theory: in every closed $3$-manifold, we construct anisotropic min-max sequences within fixed isotopy classes whose limits are stable anisotropic minimal surfaces, smooth away from at most one point. If the integrand satisfies either an ellipticity bound or a $C^3$-pinching condition, we remove this last singular point by proving two independent removable singularity theorems for anisotropic minimal surfaces that are smooth and stable away from finitely many points. These theorems also allow us to remove the singularities arising in the anisotropic Almgren-Pitts min-max construction of Theorem~\ref{1}, as well as in its multiparameter variants.
\medskip

For the remainder of the lecture, we outline the proof strategy for Theorem~\ref{1} and Theorem~\ref{2}. Actually, in \cite{DPDR,DPDRL}, we cannot prove Theorem~\ref{1} and Theorem~\ref{2} directly; instead, we first establish the existence of closed hypersurfaces with \emph{non-zero} constant anisotropic mean curvature (CMC), see Theorem~\ref{t:existenceCMC} and Theorem~\ref{thm:main3}. For the sake of exposition, we will discuss the case of zero anisotropic mean curvature, until we hit the major obstruction, and we will explain how to solve it by constructing closed CMC hypersurfaces.
\subsection{Setup of the variational problem}
A direct minimization (Plateau-type) approach cannot find a closed anisotropic minimal hypersurface, since, if $M$ is topologically a sphere, any closed hypersurface is homotopic to a point in $M$. Hence, we need to employ a min-max scheme using \emph{sweep-outs} on $M$.

Roughly speaking, a sweep-out on $M$ is a one-parameter family of surfaces that ``fills'' the manifold $M$ in a continuous way, without concentration of mass. More precisely
a sweep-out on $M$ is a continuous map $\Gamma: t\in [0, 1] \to \Gamma_t \in \mathcal{Z}^0_{n-1}(M; \mathbb{Z}_2)$, where $\mathcal{Z}^0_{n-1}(M; \mathbb{Z}_2)$ is the space of modulo two $(n-1)$-cycles on $M$ in the connected component containing $0$, endowed with the topology induced by the flat metric $\mathcal{F}$, satisfying  the following conditions:
  \begin{enumerate}
  \item There exists a continuous map $\Omega: t\in [0,1] \to \Omega_t$ valued in finite perimeter sets, such that $\Omega_0 = \emptyset$, $\Omega_1 = M$, and $\Gamma_t = \partial^* \Omega_t$ for all $t \in [0,1]$.
  \item There is no concentration of mass, meaning that
    $$
      \lim_{r \to 0} \sup\bigl\{\mathbf{F}(\Gamma_t \cap B(p,r)) \mid t \in [0, 1], p \in M\bigr\} = 0\,.
   $$
  \end{enumerate}
The collection of all sweep-outs of $M$ will be denoted by $\mathcal{L}$.
We consider the variational problem
$$
\inf_{\Gamma\in\mathcal L}\left[\max_{t\in[0,1]} \mathbf{F}(\Gamma_t)\right]=m.
$$
A sequence of sweep-outs $\{\Gamma^k\}_{k\in\mathbb N}\subset\mathcal L$ is minimizing if
$$
\lim_{k\to\infty}\max_{t\in[0,1]} \mathbf{F}(\Gamma_t^k)\;=\;m\, .
$$
A sequence $\{\Gamma_{t_k}^k\}$ is a min-max sequence if $\{\Gamma^k\}_{k\in\mathbb N}$ is minimizing and
$$\lim_{k\to\infty} \mathbf{F}(\Gamma_{t_k}^k)=m.$$
The proof strategy is to find a min-max sequence $\{\Gamma_{t_k}^k\}$ such that
 $\Gamma_{t_k}^k\rightharpoonup V$ in the sense of varifolds, with $\delta_F V=0$;
 then to show $V$ is a rectifiable varifold; then to prove $V$ is an integral varifold; and finally to conclude the optimal regularity for $V$.
\subsection{Construction of the candidate $F$-stationary varifold}
We will skip the construction of a min-max sequence which converges to an $F$-stationary varifold, as it is based on a classical \emph{pull-tight} procedure, due to Almgren and Pitts.
This is conceptually similar to finding a saddle point (a mountain pass) by gradually tightening the paths via a Palais-Smale deformation.
\medskip

The pull-tight procedure produces an $F$-stationary varifold $V$ in $M$ which is the candidate anisotropic minimal hypersurface. However, at this stage $V$ is merely an $F$-stationary varifold. The remaining steps are devoted to showing $V$ is optimally regular. Stationarity is not enough. Even for $F\equiv 1$, the optimal regularity of stationary $(n-1)$-varifolds is a major open problem, which in any case cannot be better than proving that the singular set has dimension $n-2$, without further assumptions. To overcome this issue, we need to enforce some stability properties for $V$, or more precisely an index bound.
To this aim, we will make use of the \emph{replacement property}. The idea of a \emph{replacement} is borrowed from Almgren-Pitts' regularity theory. First, let us define a replacement properly.
\begin{defin}\label{def:replac}
 Given a varifold \(V\) and  a compact set \(K \subset M\),  we say that a varifold \(V'\) is a replacement for \(V\) in \(K\)  if
\begin{itemize}
  \item[(a)] $V \restr (M \setminus K) = V' \restr (M \setminus K)$;
  \item[(b)] $  \mathbf{F}(V) =  \mathbf{F}(V') $;
  \item[(c)] there exists an almost embedded (open) hypersurface \(\Sigma\subset K\)  with
 \begin{equation}\label{singgg}
 \mathcal H^{{n-3}}(\overline{\Sigma}\setminus \Sigma)=0
 \end{equation}
  which is \(F\)-stable in \(\mathrm{Int}(K)\) and such that \(V'\restr \mathrm{Int} (K)=\Sigma\).
\end{itemize}
\end{defin}
Roughly speaking, we can prove the following replacement property for $V$. We sacrifice some details in the statement in favor of clarity. In the following we denote the annular regions $A_{x}(s,t) := B(x,t)\setminus \overline{B(x,s)}$.
\begin{prop}[Replacement property]\label{prop:replace}
There is a min-max sequence $\Gamma^k$ and a function $r: M \to (0,\infty]$ such that:
 $\Gamma^k \rightharpoonup V$, $\delta_F V=0$ and for every $x\in M$, $0 < s < t < r(x)$ there exists a replacement $V'$ for $V$ in $A_x(s,t)$.
Moreover the same replacement property holds for $V'$.
\end{prop}
As previously mentioned, for $F\equiv 1$ the proof of the replacement property is due to Almgren and Pitts \cite{Pitts}. This property has then been revisited by Colding and De Lellis \cite{CDL} and by De Lellis and Tasnady \cite{DLT}.
\begin{proof}[Rough proof idea.]
Consider a min-max sequence $\Gamma^k_{t_k}=\partial \Omega^k_{t_k}$.
Fix $x\in M$ and $0 < s < t < r(x)$.
For every $k$, choose a minimizer $\partial^* \tilde{\Omega}^k$ for $\mathbf{F}$ among boundaries of sets $Q$ for which
there exists $\{\Omega_\tau\}_{\tau\in[0,1]}$ satisfying
$\Omega_0=\Omega^k_{t_k}$,
$\Omega_1=Q$,
$\Omega_\tau\setminus A_x(s,t) =\Omega^k_{t_k} \setminus A_x(s,t)$ for all $\tau\in[0,1]$,
$\mathbf{F}(\partial^* \Omega_\tau)\leq \mathbf{F}(\partial^* \Omega^k_{t_k})+\frac{1}{k}$ for all $\tau\in[0,1]$.
Then, prove that in balls that are small enough, $\tilde{\Omega}^k$ solves the anisotropic Plateau problem.
Finally, use this fact to deduce that $\partial^* \tilde{\Omega}^k\rightharpoonup V'$, where $V'$ is the desired replacement.
\end{proof}
\subsection{Rectifiability of the limit varifold}
The next step is to prove that $V$ is a rectifiable varifold. Thanks to Theorem~\ref{thm:rect}, and since $F$ satisfies the atomic condition and $\delta_FV=0$, we just need to show that  $\Theta^{n-1}_*(x,V) > 0$ for $\|V\|$-a.e. $x$.
 Fix $x\in \operatorname{spt} \|V\|$, and $2r<r(x)$. By Proposition \ref{prop:replace}, replace $V$ with $V'$ in $A_x (r,2r)$. By the maximum principle \cite{DPDRH}, we deduce $V'\restr A_x (r,2r)$ is a {\em non-empty} optimally regular $F$-stable surface $\Sigma$. Otherwise  $V\restr B(x,r)$ would be an $F$-stationary varifold entirely contained in a ball. Intuitively, a smooth anisotropic minimal surface cannot abruptly end inside the manifold without boundary conditions, as a consequence of the maximum principle. Therefore, the replacement $V'$ must be nontrivial in $A_x (r,2r)$.

Now choose a point $y\in \partial B(x,3r/2)\cap \Sigma$. Since $\Sigma$ is $F$-stationary and stable, we can apply the lower density estimates for $F$-stationary and stable hypersurfaces proved by Allard \cite{All2}. We deduce that there exists $C>0$ depending only on $F$ and $n$ such that
$$\mathcal{H}^{n-1}(\Sigma\cap B(y,r/2))\geq C(r/2)^{n-1}.$$
In particular
$$\frac{\|V\|(B(x,2r))}{(2r)^{n-1}}\geq \frac{C\mathcal{H}^{n-1}(\Sigma\cap B(y,r/2))}{(2r)^{n-1}}\geq \frac{C(r/2)^{n-1}}{(2r)^{n-1}}\geq C.$$
Passing to the liminf in $r$, by the arbitrariness of $x$, we deduce that there exists $C>0$, depending only on $F$ and $n$ such that $\Theta^{n-1}_* (x, V)\geq C$ for any $x\in \operatorname{spt} \|V\|$, as claimed.

\subsection{Integrality of the limit varifold}
The next goal is to show $V$ is an integral varifold, that is, its multiplicity is integer-valued $\mathcal{H}^{n-1}$-a.e. on its rectifiable support.
Since $\Theta^{n-1} (x, V)>0$ for $\|V\|$-a.e. $x$, a standard strategy is to invoke our compactness result \cite[Theorem 4.1]{DR2}: if one has a sequence of integral varifolds with uniformly bounded masses and anisotropic first variations, then any limit varifold with positive density is also integral. In our situation, each $\Gamma^k$ in the min-max sequence is an integral varifold. However, we cannot guarantee a uniform bound on $\delta_F \Gamma^k$. We argue instead in the following way.

Fix $x\in \operatorname{spt} \|V\|$ such that the tangent cone $C$ is unique and is a plane $\pi$ with constant multiplicity $\theta$, that is, $C=\theta|\pi|$. We recall that this is true for $\|V\|$-a.e. $x$, given that $V$ is rectifiable, cf. \cite{SimonBook}.
We need to prove that $\theta$ is an integer value.
We observe that $\delta_FC=0$. Consider a sequence $\rho_j\downarrow 0$ such that $(\eta^{x,\rho_j})_\# V \rightharpoonup C$. Replace $V$ by $V'_j$ in $A_x( \rho_j/4, 3\rho_j/4)$ and set $W'_j= (\eta^{x,\rho_j})_\# V'_j$. After possibly passing
to a subsequence, we can assume that $W'_j  \rightharpoonup C'$, where $C'$ satisfies again $\delta_FC'=0$. A trivial consequence of the definition of replacement is that
\begin{equation}\label{CC'}
\mbox{$C'=C=\theta|\pi|$ in $B(0,1/4)\cup A_0(3/4, 1)$.}
\end{equation}
Moreover, by definition of replacement and rescaling properties, the restrictions of $W'_j$ to $A_0(1/4, 3/4)$ are the varifolds associated with optimally regular, almost embedded, $F$-stable hypersurfaces. Hence, again by compactness of $F$-stable hypersurfaces \cite{All2}, we conclude that $W'_j$ converge locally smoothly (with integer multiplicity) to some optimally regular embedded $F$-stable hypersurface $\Sigma'$ in $A_0(1/4, 3/4)$.

Here it is important to notice a major difference between the case $n=3$ and the general-dimensional case. For $n=3$ replacements are completely smooth in $A_0(1/4, 3/4)$, hence there is no need of uniform upper density estimates to apply \cite{All2}. On the contrary, for the general-dimensional case, the presence in the replacement of a singular set as in \eqref{singgg} forces us to prove a uniform upper density bound. As this obstruction will also show up in the proof of optimal regularity, we will discuss how to obtain uniform upper density estimates at the end of this section.
Since $\delta_FC'=0$, the maximum principle \cite{DPDRH} implies that
\begin{equation}\label{cond}
\overline{\Sigma'}\cap \partial B(0,3/4) \subset \pi.
\end{equation}
 We claim that $\Sigma'\subset \pi$. Indeed,  assume by contradiction the claim is false. Up to rotation, we can assume $\pi=\operatorname{span}\{e_1,\dots, e_{n-1}\}$. Moreover, without loss of generality we can assume $\Sigma' \cap \{\langle z, e_n\rangle >0\}\neq \emptyset$.
There exists $\max\{a>0: \Sigma' \cap \{\langle z, e_n\rangle =a\}\neq\emptyset\}$. By the classical maximum principle (cf. for instance \cite[Corollary 5.1]{SimonBook}):
$$\{\langle z, e_n\rangle =a\}\cap A_0(1/4, 3/4) \subset \Sigma'\cap A_0(1/4, 3/4),$$
 which contradicts \eqref{cond}.
Recalling \eqref{CC'}, this implies that $\operatorname{spt} (\|C'\|) \subset \pi$.  Since $\delta_FC'=0$, the Constancy Theorem \cite[Proposition 5]{DPDRH}, together with \eqref{CC'}, implies that  $C'=\theta|\pi|$. Since $\Sigma'$ has integer multiplicity in $A_0(1/4, 3/4)$, we conclude that $\theta$ is an integer.
\subsection{Optimal regularity of the limit varifold}
As a next step, we would like to prove optimal regularity in punctured balls. If unique continuation for $F$-stationary varifolds were true, we could simply make a replacement in a small annulus centered at a given point and deduce by unique continuation that $V$ coincides with the optimally regular replacement inside the annulus. By making the inner radius of the annulus arbitrarily small, we would deduce the regularity in the punctured ball.

Unfortunately this naive argument does not work, as unique continuation does not hold for stationary varifolds. However, Almgren and Pitts developed a scheme to make this idea work, by constructing consecutive replacements in concentric annuli.
Intuitively, if one is able to show that any two consecutive replacements in two concentric non-disjoint annuli glue smoothly, then unique continuation applies, and an easy application of the maximum principle shows that all consecutive replacements coincide with the original varifold $V$ in the punctured ball.

Proving the smooth gluing of two consecutive replacements in the isotropic setting $F\equiv 1$ is achieved by means of the monotonicity formula. Indeed, for $F\equiv 1$, the blow-ups at the points of the interface of the two consecutive replacements are proved to be hyperplanes, by simple mass comparison at different scales.
However, for a general anisotropic integrand this strategy does not work, due to the lack of a monotonicity formula. Actually, it was not even clear a priori whether a blow-up exists at a point of the interface, that is, whether upper density estimates hold at the points of the interface of two consecutive replacements.
The formation of high multiplicity seemed to be an obstruction to smoothness.
However, higher multiplicity occurs only on a {\em negligible set} if one constructs hypersurfaces with {\em non-zero} constant anisotropic mean curvature. This idea was inspired by the construction of Zhou and Zhu \cite{ZZ1,ZZ2} of constant mean curvature min-max hypersurfaces in the isotropic case, $F\equiv 1$.

This is why we decided to first prove in \cite{DPDR} and \cite{DPDRL}, respectively, the analogues of Theorem~\ref{1} and Theorem~\ref{2}, for non-zero constant anisotropic mean curvature hypersurfaces:
\begin{theorem}[\cite{DPDR}]\label{t:existenceCMC}
 If $n=3$ and $c\in \R\setminus \{0\}$, then there is a nontrivial surface $\Sigma\subset M$ without boundary which has constant anisotropic mean curvature $c$ with respect to $F$. Moreover there exists at most one singular point $p\in M$ for $\Sigma$, i.e., $\Sigma$ is $C^2$ almost embedded away from $p$. \end{theorem}
\begin{theorem}[\cite{DPDRL}]\label{thm:main3}
 For every $n\geq 3$ and $ c\in \mathbb{R}\setminus \{0\}$, there exists a smooth almost embedded hypersurface $\Sigma^{n-1}\subset M$ (that is, an immersed hypersurface which is locally a union of embedded sheets touching tangentially) with constant anisotropic mean curvature $c$ with respect to $F$, such that $\operatorname{Sing}(\Sigma) \equiv \overline{\Sigma}\setminus \Sigma$ has $\mathcal{H}^{n-3}(\operatorname{Sing}(\Sigma)) = 0$.
\end{theorem}
Theorem~\ref{t:existenceCMC} and Theorem~\ref{thm:main3} imply respectively Theorem~\ref{1} and Theorem~\ref{2} via the following argument:
\begin{itemize}
\item Construct $\Sigma_k$ with $c=1/k$.
\item $F$-stability allows one to prove that $\Sigma_k \to \Sigma$ smoothly, with $\delta_F \Sigma=0$.
\end{itemize}
Actually the second bullet point requires stronger versions of Theorem~\ref{t:existenceCMC} and Theorem~\ref{thm:main3}, in order to have some uniform curvature control in passing $\Sigma_k$ smoothly to the limit. We are going to omit these refined versions of the results for simplicity of exposition.

As in the proof of integrality of $V$, for the general-dimensional case we need to prove a uniform upper density bound in order to be able to show the required compactness for the hypersurfaces $\Sigma_k$ by means of $F$-stability. Indeed, since $\mathcal{H}^{n-3}(\operatorname{Sing}(\Sigma_k)) = 0$, Allard compactness for anisotropic stable hypersurfaces \cite{All2} requires a uniform upper density bound. We postpone the proof idea of the uniform density ratio bound to the end of this section.

Hence, we are left to prove Theorem~\ref{t:existenceCMC} and Theorem~\ref{thm:main3}. As the proof strategy in the positive constant anisotropic mean curvature setting remains the same as the one presented for anisotropic minimal hypersurfaces, we will just focus on the smooth gluing of two consecutive replacements in two concentric annuli.
The main advantage, as already mentioned above, is that as in the work of Zhou and Zhu \cite{ZZ1,ZZ2}, we construct replacements with positive constant anisotropic mean curvature with multiplicity one, away from an $(n-2)$-dimensional set with multiplicity two, and with singular set of zero $\mathcal{H}^{n-3}$-measure.
\begin{proof}[Proof sketch of smooth gluing.]
We prove smooth gluing of two consecutive replacements in two concentric annuli $ A'$, $ A''$ at their interface, that is, the larger connected component of $\partial  A''$, which we will denote in the following by $\partial^+  A''$. We denote the replacements by $V'$ and $V''$ and their optimally regular $F$-stable components respectively by $\Sigma'=\|V'\|\restr A'$ and $\Sigma''=\|V''\|\restr A''$. We will distinguish below between the two arguments used for $n=3$ in \cite{DPDR} and for the general dimension in \cite{DPDRL}.

{\bf The case $n=3$.}
We recall that for $n=3$, the surfaces $\Sigma'$ and $\Sigma''$ have no singular points.
We first fix $y\in \partial^+  A''$ of multiplicity one for $V'$.
We prove that the min-max sequence $ \tilde \Gamma^k$ approximating $V''$ is regular in $ A''$ up to $\partial^+  A''$. We achieve this in two steps:
\begin{itemize}
\item  For every $k\in \mathbb N$ we show that $\tilde \Gamma^k\cap  A''$ has at least one blow-up consisting of a half-plane. The proof borrows ideas from the work of Hardt \cite[Lemma 4.5]{Hardt} and of De Philippis and Maggi \cite[Lemma 5.4]{DPM}. This implies that $\Theta^{2}_*(y,\tilde \Gamma^k\cap  A'')\leq \frac 12$.
\item  Since $\tilde \Gamma^k$ solves the anisotropic Plateau problem among boundaries of finite perimeter sets in small enough balls, then regularity at $y$ follows from the work of Duzaar and Steffen \cite{DuzSte}.
\end{itemize}
By the $F$-stability of $\tilde \Gamma^k \cap  A''$, the work of White \cite{W} provides uniform boundary curvature estimates for $\tilde \Gamma^k\cap  A''$. Here we strongly use $n=3$ to obtain density estimates at all relevant scales by the extended monotonicity formula, as the exponent in the stability inequality
$$\int |A|^2\varphi^2\leq C\int |D\varphi |^2$$
is critical for $n=3$.
A graphicality argument allows one to conclude that $\Sigma'$ and $\Sigma''$ glue smoothly at $y$.
We now have to analyze points $y\in \partial^+ A''$ of multiplicity two for $V'$. By the coarea formula, for a.e. choice of the larger radius of the annulus $ A''$, these points of multiplicity two are isolated.
We fix such a point $y$.

As we need to perform a blow-up argument, we first show a density ratio upper bound at $y$ for the second replacement by means of \cite{DPDRH}. Hence there exists a blow-up $\mathbf C$ at $y$ for $V''$.
By our previous analysis at points in $\partial^+  A''$ of multiplicity one for $V'$, we have that
 $$\mathbf C  \restr A_0(\alpha,1/\alpha)\geq 2T_y\Sigma'  \restr A_0(\alpha,1/\alpha), \qquad \forall \alpha\in (0,1),$$
 where $\Sigma'$ is the first replacement. Hence $\mathbf C\geq 2T_y\Sigma'$.
Since both $\mathbf C$ and $2T_y\Sigma'$ are $F$-stationary, we deduce that $\mathbf C':=\mathbf C-2T_y\Sigma' $ is $F$-stationary. With more work, borrowing ideas from \cite{Hardt,DPM}, we also show that $\mathbf C'$ is contained in a wedge. By the maximum principle \cite{DPDRH}, we deduce that $\mathbf C'=0$, hence $\mathbf C=2T_{y} \Sigma'$.
By a graphicality argument, again we conclude that the two replacements glue smoothly.

{\bf The general case $n\geq 3$.}
The main obstruction in extending the $3$-dimensional argument to every dimension is that, for $n>3$, we have no boundary curvature estimates at a multiplicity one point $y \in \partial^+  A''$ for $V'$.
To solve this issue, our idea is to perform a third replacement $V'''$ around $y$, in an annulus bounded by spherical caps meeting at a sufficiently small angle.
More precisely, fix a point $y\in \operatorname{Reg}(\Sigma')\cap \partial^+  A''$. In particular $y$ is of multiplicity one for $V'$.  Fix $0<\varepsilon\ll 1$ to be chosen later. There exists a sufficiently small radius $r>0$ such that for every $z\in  \operatorname{Reg}(\Sigma')\cap B_r (y)$,
  \[
  \dist (T_z\Sigma',T_y\Sigma')<\varepsilon.
  \]
We fix a point $z\in \operatorname{Reg}(\Sigma')\cap B_r (y) \setminus A''$. We consider a convex domain $\tilde C$ bounded by the union of two spherical caps with the same boundary. This common boundary is an $(n-2)$-dimensional sphere $S$ centered at $y$, with $z \in S$ and $T_zS = T_z(\Sigma' \cap \partial \tilde C)$. Moreover, the two caps intersect at an angle \(3\varepsilon\). We denote by
 $C$  the annulus obtained by removing from $\tilde C$ a translation of $\frac 12 \tilde C$, that is concentric with $\tilde C$.
Let  $V'''$ be a replacement of  $V''$ inside $C$. Since $V'''$ coincides with $V''$ outside $C$, we know that the family $\{(\eta^{z,r'})_\# V'''\}_{r'<r}$ has uniformly bounded mass and consequently ${\rm Tan}(z, V''')\neq \emptyset$.
We can choose the angle  \(\varepsilon\)  sufficiently small such that there exists a blow-up varifold $W\in {\rm Tan}(z,V''')$ which satisfies $W=T_z\Sigma'$. This is obtained by first proving, via the maximum principle, that $W$ consists of finitely many half-hyperplanes. Then the desired rigidity follows from $F$-stationarity and the smallness of the angle $\varepsilon$.
As a consequence we can show that $V'$ and $V'''$ glue smoothly at $z$. In particular, \(z\) is a multiplicity one point for \(V'''\).  By unique continuation, there exists a connected component of $V'''\restr C$ which coincides with $\Sigma'\cap C$.
By the maximum principle \cite{DPDRH} we can then deduce that
$$
    \Sigma'\cap \partial \tilde  C\subset \mbox{spt}(V''') \cap \partial \tilde C \subset \mbox{spt}(V'''\restr (M\setminus \tilde C)) \subset \mbox{spt}(V''\restr (M\setminus \tilde C)).
$$
  Hence
\begin{equation}\label{usethis}
    \Sigma'\cap \partial \tilde  C \subset \mbox{spt}(V'')\,.
  \end{equation}
By the arbitrariness of the point $z\in \operatorname{Reg}(\Sigma')\cap B_r (y) \setminus A''$, we can foliate $\tilde C$ with a $1$-parameter family $\{\partial \tilde  C_\alpha\}_{\alpha \in [0,1]}$ and, by \eqref{usethis} applied to each $\partial \tilde  C_\alpha$, we conclude that
  \[
    \Sigma'\cap \tilde C \subset \mbox{spt}(V'').
    \]
    This combined with the maximum principle implies that every $Z\in {\rm Tan}(y,V'')$ coincides with $T_y \Sigma'$.
    A graphicality argument then allows one to conclude the desired smooth gluing at $y$.

    The smooth gluing at points of multiplicity two for $V'$ in $\partial^+  A''$ is similar to the three-dimensional argument, and we omit the details.

\end{proof}
By the previous argument, we deduce the optimal regularity of $V$ in punctured balls. In particular,  $V$ is optimally regular away from finitely many points.
While this concludes the min-max construction for $n>3$, as finitely many points have zero $\mathcal{H}^{n-3}$-measure, for $n=3$, we are still left to remove these finitely many singularities.

Due to the lack of a monotonicity formula we cannot remove the finitely many singularities just using the $F$-stationarity and stability, as these singularities may have infinite density a priori. Hence we cannot extend the stability inequality across these points.
However, a combinatorial lemma of Almgren-Pitts shows that the map of radii $r$ for the existence of replacements in Proposition \ref{prop:replace} satisfies one of the following:
\begin{itemize}
\item[(i)] there exists $R>0$ such that $r(x)\equiv R$;
\item[(ii)] there exists $p\in M$ such that $r(x)=d(x,p)$ for $x\neq p$ and $r(p)=\infty$.
\end{itemize}
So we can cover $M\setminus \{p\}$ with finitely many punctured balls $B(x,r(x))\setminus \{x\}$.
We cannot remove the singularity $p$ by further exploiting the Almgren-Pitts combinatorial lemma, as $p$ accounts for the index of the constructed min-max anisotropic minimal surface.

In order to remove the last singular point $p$ for $n=3$, in \cite{DPDRL} we prove a uniform density ratio upper bound for $\Sigma$. This allows:
\begin{itemize}
\item to extend the stability inequality (hence the $L^2$-curvature estimates) across the singular point $p$,
\item to prove that the blow-up at $p$ is unique and is a plane \cite{W},
\item to remove the singular point $p$.
\end{itemize}
We recall that we have also used a uniform density ratio upper bound for the compactness of stable surfaces with a singular set of zero $\mathcal{H}^{n-3}$-measure. Obtaining this upper density bound is one of the major contributions of \cite{DPDRL}. 

Below, we briefly sketch the strategy to prove a uniform density ratio upper bound:
\begin{itemize}
\item We build an optimal nested volume-parametrized sweep-out $\partial \Omega_t$ as in the work of Chambers and Liokumovich \cite{CL},  and of Chodosh, Liokumovich and Spolaor \cite{CLS}.
\item We fix a sequence of triangulations $\{T_k\}$ of $M$ such that $\operatorname{diam}(\sigma)\approx 2^{-k}$ for $\sigma \in T_k$. We aim to get a new optimal nested sweep-out $\partial \tilde \Omega_t^*$ such that
\begin{equation}\label{uu}
\mathcal H^{n-1}(\partial \tilde \Omega_t^* \cap \sigma)\lesssim (2^{-k})^{n-1}, \qquad  \forall \sigma \in T_k,  \forall t\in [0,1].
\end{equation}
\item By lower semicontinuity, for every $k\in \mathbb N$ and $\sigma \in T_k$, the set of ``bad'' slices of $\partial \Omega_t$ is an at most countable union of disjoint open intervals
\[
    \bigcup^\infty_{i = 1} (a_i, b_i)\,.
\]
\item For every $i\in \mathbb N$ we modify all slices $\partial \Omega_t$ for $t\in (a_i,b_i)$, by solving a suitable anisotropic Plateau-type problem. This provides the desired \eqref{uu} by a simple comparison argument.
\end{itemize}

\vspace{0.5cm}

\small{
Antonio De Rosa,\\
\emph{Department of Decision Sciences and BIDSA, Bocconi University, Milan, Italy,}\\
\emph{Email address:} antonio.derosa@unibocconi.it}
\end{document}